\documentclass{amsart}

\usepackage{amssymb,amsthm}

\usepackage{pstricks,multido}
\usepackage{pst-node}
\usepackage{a4wide}

\usepackage{hyperref}

\title[New interpretations for noncrossing partitions]
{New interpretations for noncrossing partitions of classical types}
\author{Jang Soo Kim}
\thanks{The author was supported by the grant ANR08-JCJC-0011.}
\email{kimjs@math.umn.edu}
\subjclass[2000]{05A15; 05E15}
\keywords{noncrossing partitions, nonnesting partitions, type-preserving bijections, Catalan tableaux}
\date{\today}
\newtheorem{thm}{Theorem}[section]
\newtheorem{lem}[thm]{Lemma}
\newtheorem{prop}[thm]{Proposition}
\newtheorem{cor}[thm]{Corollary}
\theoremstyle{definition}
\newtheorem{defn}{Definition}[section]

\newtheorem{question}{Question}[section]
\newtheorem{example}{Example}[section]
\newtheorem*{notation}{Notation}
\theoremstyle{remark}
\newtheorem{remark}{Remark}[section]

\psset{unit=0.5cm, dotsize=5pt, linewidth=0.1pt}

 \def\tsframe(#1,#2)(#3,#4){}
\def\vput#1{\pnode(#1,1){#1} \pscircle*(#1,1){.1} \rput(#1,.5){#1}}
\def\vvput#1#2#3{\pnode(#1,1){#2} \pscircle*(#1,1){.1} \rput(#1,.5){#3}}
\def\edge#1#2{\ncarc[arcangle=50]{#1}{#2}}
\def\EDGE#1#2{\ncarc[arcangle=30]{#1}{#2}}
\def\hcedge#1#2{\rput(#2,1){\psarc(#1,0){#2}{0}{180}}}

\def\activevertex#1[#2]{\pscircle[linecolor=red](#1,1){.2} \rput(#1,2){\bf \blue
    #2}}

\def\tedge#1#2#3{\psset{linewidth=1pt,linecolor=#3}
\edge{#1}{#2} \pscircle*(#1,1){.15}\pscircle*(#2,1){.15}}

\newcount\ax \newcount\ay
\newcount\bx \newcount\by
\newcount\cx \newcount\cy
\newcount\dx \newcount\dy
\def\cell(#1,#2)[#3]{
\ax=#2 \ay=#1
\multiply\ay by-1
\bx=\ax \by=\ay
\cx=\ax \cy=\ay
\dx=\ax \dy=\ay
\advance\bx by-1
\advance\dy by1
\advance\cx by-1
\advance\cy by1
\psline (\dx,\dy)(\ax,\ay)(\bx,\by)
\rput(\number\cx.5,
\ifnum\cy=0 -0.5\else\number\cy.5\fi){#3}
}

\def\colnum[#1,#2]{\rput[b](#1.5,.2){$#2$}}
\def\rownum[#1,#2]{\rput[r](-.2,-#1.5){$#2$}}

\def\halfcircle#1#2 {
  \rput(#1,0){
    \psarc(2,0){2}{0}{180} \psline(0,0)(4,0)
    \rput (2,1) {#2}
 }
}

\def\rect#1#2 {
  \rput(#1,0){
    \psframe[framearc=.8](0.5,0)(3.5,2)
   \rput (2,1) {#2}
 }
}

\def\connect(#1,#2){
  \rput(#1,0){
    \psarc(#2,0){#2}{0}{180}
 }
}

\def\nonnested#1#2#3{
  \rput(#1,0){
 \halfcircle0{#2} \connect(4,3) \rect5{#3}    
}
}
\newcommand\ovxi{\overline{\xi}}

\newcommand{\type}{\operatorname{type}}

\newcommand{\ncbb}{\mathfrak{B}}
\newcommand{\ncdd}{\mathfrak{D}}

\newcommand{\LP}{\operatorname{LP}}
\newcommand{\ovLP}{\overline{\operatorname{LP}}}
\newcommand{\nca}{\operatorname{NC}}
\newcommand{\ncb}{\operatorname{NC}_B}

\newcommand{\ncd}{\operatorname{NC}_D}
\newcommand{\nna}{\operatorname{NN}}
\newcommand{\nnb}{\operatorname{NN}_B}
\newcommand{\nnc}{\operatorname{NN}_C}
\newcommand{\nnd}{\operatorname{NN}_D}
\newcommand{\cta}{\operatorname{CT}}
\newcommand{\ctb}{\operatorname{CT}_B}
\newcommand{\ctd}{\operatorname{CT}_D}

\newcommand{\floor}[1]{\left\lfloor #1 \right\rfloor}

\newcommand{\nal}{\operatorname{na}}
\newcommand{\nne}{\operatorname{nn}}
\newcommand{\nabk}{\operatorname{NABK}}
\newcommand{\nnbk}{\operatorname{NNBK}}

\newcommand{\ncmap}{\phi^{\operatorname{NC}}}
\newcommand{\nnmap}{\phi^{\operatorname{NN}}}
\newcommand{\ovncnn}{\overline{\operatorname{NC}}^{\operatorname{NN}}}
\newcommand{\ncnn}{\operatorname{NC}^{\operatorname{NN}}}
\newcommand{\ncna}{\operatorname{NC}^{\operatorname{NA}}}
\newcommand{\nnna}{\operatorname{NN}^{\operatorname{NA}}}
\newcommand{\ncnnd}{\ncnn_{\{0,\pm1\}}}
\newcommand{\ncnad}{\ncna_{\{0,\pm1\}}}
\newcommand{\nnnad}{\nnna_{\{0,\pm1\}}}

\newcommand{\NC}{\operatorname{NC}}
\newcommand{\ovrho}{\overline{\rho}}

\begin{document}

\begin{abstract}
  We interpret noncrossing partitions of type $B$ and type $D$ in terms of
  noncrossing partitions of type $A$. As an application, we get type-preserving
  bijections between noncrossing and nonnesting partitions of type $B$, type $C$
  and type $D$ which are different from those in the recent work of Fink and
  Giraldo. We also define Catalan tableaux of type $B$ and type $D$, and find
  bijections between them and noncrossing partitions of type $B$ and type $D$
  respectively.
\end{abstract}

\maketitle
%\tableofcontents

\section{Introduction}
\label{sec:introduction}

A {\em partition} of a set $U$ is a collection of mutually disjoint nonempty
subsets of $U$, called \emph{blocks}, whose union is equal to $U$. Let $\Pi(n)$
denote the set of partitions of $[n]=\{1,2,\ldots,n\}$.  For $\pi\in\Pi(n)$, an
\emph{edge} of $\pi$ is a pair $(i,j)$ of integers $i$ and $j$ with $i<j$ such
that $i$ and $j$ are in the same block of $\pi$ and this does not contain any
integer between them.

A partition $\pi\in\Pi(n)$ is called \emph{noncrossing}
(resp.~\emph{nonnesting}) if $\pi$ does not have two edges $(a,b)$ and $(c,d)$
satisfying $a<c<b<d$ (resp.~$a<c<d<b$). We denote by $\nca(n)$ (resp.~$\nna(n)$)
the set of noncrossing (resp.~nonnesting) partitions of $[n]$.

Recently, noncrossing and nonnesting partitions have received great attention
and have been generalized in many different ways both combinatorially and
algebraically; we refer the reader to excellent expositions \cite{Armstrong,
  Simion2000} and the references therein. Bessis \cite{Bessis2003}, Brady and
Watt \cite{Brady2008} defined the set $\NC(W)$ of noncrossing partitions for
each finite reflection group $W$ where $\NC(A_{n-1})$ is the same as $\nca(n)$.
Postnikov defined the set $\nna(W)$ of nonnesting partitions for each
crystallographic reflection group $W$ where $\nna(A_{n-1})$ is the same as
$\nna(n)$; see \cite[Remark 2]{Reiner1997}.

For each classical reflection group $W$, we have a combinatorial model for
$\NC(W)$: the set $\ncb(n)$ of noncrossing partitions of type $B_n$ defined by
Reiner \cite{Reiner1997} and the set $\ncd(n)$ of noncrossing partitions of type
$D_n$ defined by Athanasiadis and Reiner \cite{Athanasiadis2005}. Both $\ncb(n)$
and $\ncd(n)$ are subsets of the set $\Pi_B(n)$ of partitions of type $B_n$
introduced by Reiner \cite{Reiner1997}. We also have combinatorial models for
$\nna(W)$ introduced by Athanasiadis \cite{Athanasiadis1998}, which we will
denote by $\nnb(n)$, $\nnc(n)$ and $\nnd(n)$. All of these are again subsets of
$\Pi_B(n)$.

The main purpose of this paper is to give new interpretations for $\ncb(n)$,
$\ncd(n)$, $\nnb(n)$, $\nnc(n)$ and $\nnd(n)$. To do this, we first interpret
$\pi\in\Pi_B(n)$ as a triple $(\sigma, X,Y)$, where $\sigma\in\Pi(n)$, $X$ is a
set of blocks of $\sigma$ and $Y$ is a maximal matching on $X$. As a
consequence, we obtain the following formula for the cardinality of $\Pi_B(n)$:
$$\#\Pi_B(n) = \sum_{k=1}^n S(n,k) t_{k+1},$$ where
$S(n,k)$ is the Stirling number of the second kind and $t_n$ is the number of
involutions on $[n]$.

\begin{defn}\label{def1}
  For a partition $\sigma\in\Pi(n)$, a block $B$ of $\sigma$ is called
  \emph{nonnested} (resp.~\emph{nonaligned}) if there is no edge $(i,j)$ of
  $\sigma$ with $i<\min(B)\leq\max(B)<j$ (resp.~$\max(B)<i$).  We denote by
  $\nnbk(\sigma)$ (resp.~$\nabk(\sigma)$) the set of nonnested
  (resp.~nonaligned) blocks of $\sigma$. We define
  \begin{align*}
    \ncnn(n) &= \{ (\sigma,X) : \sigma\in\nca(n), X\subset\nnbk(\sigma)\},\\
    \ncna(n) &= \{ (\sigma,X) : \sigma\in\nca(n), X\subset\nabk(\sigma)\},\\
    \nnna(n) &= \{ (\sigma,X) : \sigma\in\nna(n), X\subset\nabk(\sigma)\}.
 \end{align*}
 We denote by $\ncnnd(n)$ (resp.~$\ncnad(n)$ and $\nnnad(n)$) the set of triples
 $(\sigma,X,\epsilon)$, where $(\sigma,X)$ is in $\ncnn(n)$ (resp.~$\ncna(n)$
 and $\nnna(n)$) and $\epsilon\in\{-1,0,1\}$ with the additional condition that
 if $X=\emptyset$ then $\epsilon=0$.
\end{defn}

By using our interpretation for $\Pi_B(n)$, we obtain a bijection between
$\ncb(n)$ (resp.~$\nnb(n)$, $\nnc(n)$) and $\ncnn(n)$ (resp.~$\nnna(n)$,
$\nnna(n)$). Similarly we get a bijection between $\ncd(n)$ (resp.~$\nnd(n)$)
and $\ncnnd(n-1)$ (resp.~$\nnnad(n-1)$).  Since $\ncnn(n)$ and $\ncnnd(n-1)$
concern only type $A$ noncrossing partitions, our interpretations have the
advantage of understanding $\ncb(n)$ and $\ncd(n)$ as easily as $\nca(n)$.

To make a connection between noncrossing and nonnesting partitions in our
interpretations we find an involution on $\nca(n)$ which interchanges the
nonnested blocks and the nonaligned blocks. Thus, as a byproduct, we get that
the nonnested blocks and the nonaligned blocks have a joint symmetric
distribution on $\nca(n)$, in other words,
$$\sum_{\pi\in\nca(n)} x^{\nne(\pi)} y^{\nal(\pi)} = 
\sum_{\pi\in\nca(n)} x^{\nal(\pi)} y^{\nne(\pi)},$$ where $\nne(\pi)$
(resp.~$\nal(\pi)$) denotes the number of nonnested (resp.~nonaligned) blocks of
$\pi$.

Combining our bijections together with the bijection between $\nca(n)$ and
$\nna(n)$ due to Athanasiadis \cite{Athanasiadis1998}, we obtain type-preserving
bijections, i.e.~ bijections preserving block sizes, between noncrossing and
nonnesting partitions of classical types. Our type-preserving bijections are
different from those of Fink and Giraldo \cite{Fink}.

We provide another interpretation for $\ncb(n)$ and $\ncd(n)$: a bijection
between $\ncb(n)$ and the set $\ncbb(n)$ of pairs $(\sigma,x)$ where
$\sigma\in\nca(n)$ and $x$ is either $\emptyset$, an edge of $\sigma$ or a block
of $\sigma$, and a bijection between $\ncd(n)$ and the set $\ncdd(n)$ of pairs
$(\sigma,x)$ where $\sigma\in\nca(n-1)$ and $x$ is either $\emptyset$, an edge
of $\sigma$, a block of $\sigma$ or an integer in $[\pm(n-1)]$. In fact,
$\ncbb(n)$ and $\ncdd(n)$ are essentially the same as $\nca(n)\times[n+1]$ and
$\nca(n-1)\times[3n-2]$ respectively. Using these interpretations, we give
another proof of the formula for the number of noncrossing partitions of type
$B_n$ and type $D_n$ with given block sizes.

It is well known that $\nca(n)$ is in bijection with the set of Dyck paths,
i.e.~lattice paths from $(0,0)$ to $(n,n)$ which do not go below the line
$y=x$. Using $\ncna(n)$ and $\ncnad(n-1)$ we find a bijection between $\ncb(n)$
and the set $\LP(n)$ of lattice paths from $(0,0)$ to $(n,n)$ and a bijection
between $\ncd(n)$ and the set $\ovLP(n)$ of lattice paths in $\LP(n)$ which do
not touch $(n-1,n-1)$ and $(n,n-1)$ simultaneously.

Permutation tableaux were first introduced by Postnikov \cite{Postnikov} in the
study of the totally nonnegative Grassmannian. Catalan tableaux are special
permutation tableaux. Permutation tableaux and Catalan tableaux are respectively
in bijection with permutations and noncrossing partitions; see
\cite{Burstein2007,Corteel2009,Nadeau,Steingrimsson2007}. Lam and Williams
\cite{Lam2008a} defined permutation tableaux of type $B_n$. In this paper we
define Catalan tableaux of type $B_n$ and $D_n$ which are special permutation
tableaux of type $B_n$. Then we find bijections between them and $\ncb(n)$ and
$\ncd(n)$.

The rest of this paper is organized as follows.  In
Section~\ref{sec:preliminaries} we recall the definitions of noncrossing and
nonnesting partitions of finite reflection groups and the combinatorial models
for them for classical reflection groups. In Section~\ref{sec:part-class-types}
we define a map from $\Pi_B(n)$ to the set of certain triples. In
Section~\ref{sec:interpr-type-b} we give new interpretations for $\ncb(n)$,
$\ncd(n)$, $\nnb(n)$, $\nnc(n)$ and $\nnd(n)$. In
Section~\ref{sec:type-preserving} we find type-preserving bijections between
noncrossing and nonnested partitions of classical types. In
Section~\ref{sec:another_interpret} we find a bijection between $\ncb(n)$
(resp.~ $\ncd(n)$) and $\ncbb(n)$ (resp.~$\ncdd(n)$). In
Section~\ref{sec:latticepaths} we find a bijection between $\ncb(n)$
(resp.~$\ncd(n)$) and $\LP(n)$ (resp.~$\ovLP(n)$). In
Section~\ref{sec:catal-tabl-class} we define the sets $\ctb(n)$ and $\ctd(n)$ of
Catalan tableaux of type $B_n$ and type $D_n$, and find bijections between them
and $\ncb(n)$ and $\ncd(n)$ respectively.

\section{Preliminaries}
\label{sec:preliminaries}

In this section we recall the definitions noncrossing and nonnesting partitions
of finite reflection groups and the combinatorial models $\ncb(n)$, $\ncd(n)$,
$\nnb(n)$, $\nnc(n)$ and $\nnd(n)$.

\subsection{General definitions for noncrossing and nonnesting partitions}
For a finite Coxeter system $(W,S)$ with the set $T=\{wsw^{-1}:s\in S, w\in W\}$
of reflections, the \emph{absolute length} $\ell_T(w)$ of an element $w\in W$ is
defined to be the smallest integer $i$ such that $w$ can be written as a product
of $i$ reflections. The \emph{absolute order} on $W$ is defined as follows:
$u\leq_T w$ if and only if $\ell_T(w) = \ell_T(u) + \ell_T(u^{-1} w)$.  Then the
noncrossing partition poset $\nca(W)$ is defined to be the interval $\{w\in W:
1\leq_T w \leq_T c\}$, where $c$ is a Coxeter element. It turns out that
$\nca(W)$ does not depend on the particular choice of $c$ up to isomorphism.

Nonnesting partitions are defined for crystallographic reflection groups.
Suppose $W$ is a crystallographic reflection group and $\Phi^+$ is a positive
root system of $W$. The \emph{root poset} $(\Phi^+,\leq)$ has the partial order
$\alpha \leq \beta$ if and only if $\beta-\alpha$ can be written as a linear
combination of the positive roots with nonnegative integer coefficients.  A
\emph{nonnesting partition} of $W$ is an antichain in the root poset
$(\Phi^+,\leq)$. We denote by $\nna(W)$ the set of nonnesting partitions of $W$.

For classical types, we will use the following root posets:
\begin{align*}
  \Phi^+(A_{n-1}) &= \{e_i-e_j : 1\leq i < j \leq n\},\\
 \Phi^+(B_{n}) &= \{e_i \pm e_j : 1\leq i < j \leq n\} \cup \{e_i : 1\leq i \leq n\},\\
  \Phi^+(C_{n}) &= \{e_i \pm e_j : 1\leq i < j \leq n\} \cup \{2e_i : 1\leq i \leq n\},\\
  \Phi^+(D_{n}) &= \{e_i \pm e_j : 1\leq i < j \leq n\}.
\end{align*}

\subsection{Combinatorial models}

We use the definitions in \cite{Fink}. For type $D_n$, our definitions are
stated in a slightly different way from those in \cite{Fink}, but one can easily
check that they are equivalent.

For a partition $\pi$ of a finite set $U$ and a total order $a_1\prec a_2\prec
\cdots \prec a_n$ of $U$, the \emph{standard representation of $\pi$ with
  respect to} the order $a_1\prec a_2\prec \cdots\prec a_n$ is the drawing
obtained as follows. Arrange $a_1,a_2,\ldots,a_n$ in a horizontal line. Draw an
arc between $a_i$ and $a_j$ for each pair $(i,j)$ with $i<j$ such that
$a_i,a_j\in B$ for a block $B$ of $\pi$ which does not contain $a_t$ with
$i<t<j$. See Figure~\ref{fig:standrep}.

\begin{figure}
  \centering
  \begin{pspicture}(0,0)(11,2) 
\tsframe (0,0)(11,2) 
    \vvput{1}{4}{4}\vvput{2}{3}{3}\vvput{3}{8}8\vvput{4}{1}1\vvput{5}{5}5\vvput{6}{2}2
    \vvput{7}{6}6\vvput{8}{7}7\vvput{9}{10}{10}\vvput{10}{9}9
 \edge45 \edge56 \edge38 \edge81 \edge{10}9
 \end{pspicture}
 \caption{The standard representation of
   $\{\{1,3,8\},\{2\},\{4,5,6\},\{7\},\{9,10\}\}$ with respect to the order
   $4\prec3\prec8\prec1\prec5\prec2\prec6\prec7\prec10\prec9$.}
  \label{fig:standrep}
\end{figure}
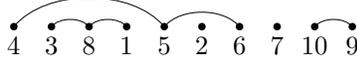

We say that $\pi$ is \emph{noncrossing} (resp.~\emph{nonnesting}) \emph{with
  respect to} the order $a_1\prec a_2\prec \cdots\prec a_n$ if $\pi$ satisfies
the following condition: if $a_i,a_k\in B$ and $a_j,a_\ell\in B'$
(resp.~$a_i,a_\ell\in B$ and $a_j,a_k\in B'$) for some blocks $B$ and $B'$ of
$\pi$ and for some integers $i<j<k<\ell$, then we have $B=B'$. In other words,
$\pi$ is noncrossing (resp.~nonnesting) with respect to the order $a_1\prec
a_2\prec \cdots\prec a_n$ if and only if the standard representation of $\pi$
with respect to this order does not have two arcs which cross each other
(resp. two arcs one of which nests the other). For example, the partition in
Figure~\ref{fig:standrep} is noncrossing but not nonnesting with respect to the
order written there.  

\begin{figure}
  \centering
\begin{pspicture}(0,0)(11,2.5)
\tsframe (0,0)(11,2.5)
  \multido{\n=1+1}{10}{\vput{\n}}
\edge14 \edge4{10} \edge56 \edge67 \edge79 \edge23
\end{pspicture}
 \caption{A noncrossing partition of type $A_{9}$.}
  \label{fig:nca}
\end{figure}
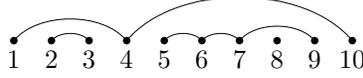

\begin{figure}
  \centering
\begin{pspicture}(0,0)(11,3)
\tsframe (0,0)(11,3)
  \multido{\n=1+1}{10}{\vput{\n}}
\hcedge11 \hcedge21 \hcedge41 \hcedge51 \hcedge6{1.5} \hcedge7{1.5}
\end{pspicture}
\caption{A nonnesting partition of type $A_{9}$.}
  \label{fig:nna}
\end{figure}
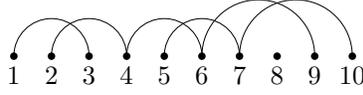
 
A \emph{noncrossing partition} (resp.~\emph{nonnesting partition}) is a
partition of $[n]$ which is noncrossing (resp.~nonnesting) with respect to the
order $1\prec 2\prec \cdots\prec n$. See Figures~\ref{fig:nca} and \ref{fig:nna}
for an example.  We denote by $\nca(n)$ (resp.~$\nna(n)$) the set of noncrossing
(resp.~nonnesting) partitions of type $A_{n-1}$.

There is a natural bijection between $\nca(A_{n-1})$ and $\nca(n)$. If we take
$c=(1,2,\dots,n)$ for the Coxeter element, each element in $\nca(A_{n-1})$ can
be written as a product of disjoint cycles of form $(a_1,a_2,\dots,a_k)$ where
$a_1<a_2<\cdots<a_k$. Then the bijection is simply changing each cycle
$(a_1,a_2,\dots,a_k)$ to the block $\{a_1,a_2,\dots,a_k\}$.  One can check that
we alway get a noncrossing partition. For example, $(1,4,10)(2,3)(5,6,7,9)(8)\in
\nca(A_{9})$ corresponds to the noncrossing partition in
Figure~\ref{fig:nca}. In fact, this bijection is a poset isomorphism if we order
$\nca(n)$ by refinement. Thus we have $\nca(A_{n-1})\cong\nca(n)$.

Similarly, there is a natural bijection between $\nna(A_{n-1})$ and
$\nna(n)$. For an antichain $\pi$ of $\Phi^+(A_{n-1})$, we construct the
corresponding nonnesting partition by making the edge $(i,j)$ for each element
$e_i-e_j\in \pi$. For example, the nonnesting partition in Figure~\ref{fig:nna}
corresponds to
\[
\{e_1-e_3,e_2-e_4,e_4-e_6,e_6-e_9,e_5-e_7,e_7-e_{10}\} \in \nna(A_9). 
\] 
Thus we have $\nna(n)\cong\nna(A_{n-1})$.

In order to define combinatorial models for noncrossing and nonnesting
partitions of other classical types, we need type $B$ partitions introduced by
Reiner \cite{Reiner1997}.  There is a natural way to identify $\pi\in\Pi(n)$
with an intersection of a collection of the following reflecting hyperplanes of
type $A_{n-1}$:
\[
\{x_i-x_j=0:1\leq i<j\leq n\}.
\]
For example, $\{\{1,3,4\},\{2,6\},\{5\}\}$ corresponds to
\[
\{ (x_1,\dots,x_6)\in \mathbb{R}^6: x_1=x_3=x_4, x_2=x_6\}.
\]

With this observation Reiner \cite{Reiner1997} defined a partition of type $B_n$
to be an intersection of a collection of the following reflecting hyperplanes of
type $B_{n}$:
\[
\{x_i=0 : 1\leq i\leq n\}\cup \{x_i\pm x_j=0:1\leq i<j\leq n\}.
\]
Note that we can also consider such an intersection as a partition of 
\[
[\pm n]=\{1,2,\ldots,n,-1,-2,\ldots,-n\}.
\]
 For example, the intersection
\[
\{ (x_1,\dots,x_8)\in \mathbb{R}^8: x_1=-x_3=x_6, x_5=x_8, x_2=x_4=0\}
\]
corresponds to
\[
\{ \pm \{1,-3,6\}, \{2,4,-2,-4\}, \pm\{5,8\}, \pm\{7\} \},
\] 
which means 
\[
\{ \{1,-3,6\}, \{-1,3,-6\}, \{2,4,-2,-4\}, \{5,8\}, \{-5,-8\}, \{7\}, \{-7\}\}.
\]
Equivalently, we define a partition of type $B_n$ as follows. 

A \emph{partition of type $B_n$} is a partition $\pi$ of $[\pm n]$ such that if
$B$ is a block of $\pi$ then $-B=\{-x:x\in B\}$ is also a block of $\pi$, and
there is at most one block, called a \emph{zero block}, which satisfies
$B=-B$. We denote by $\Pi_B(n)$ the set of partitions of type $B_n$.

Now we are ready to define combinatorial models for noncrossing and nonnesting
partitions of other classical types.

A \emph{noncrossing partition of type $B_n$} is a partition $\pi\in\Pi_B(n)$
which is noncrossing with respect to the order $1\prec 2\prec \cdots\prec n\prec
-1\prec -2\prec \cdots\prec -n$. See Figure~\ref{fig:typeb_inter} for an
example.  A \emph{noncrossing partition of type $D_n$} is a partition
$\pi\in\Pi_B(n)$ such that
\begin{enumerate}
\item if $\pi$ has a zero block $B$, then $\{n,-n\}\subsetneq B$,
\item $\pi'\in\ncb(n-1)$, where $\pi'$ is the partition obtained from $\pi$ by
  taking the union of the blocks containing $n$ or $-n$ and removing $n$ and
  $-n$.
\end{enumerate}
See Figure~\ref{fig:typeD_inter2} for an example. We denote by $\ncb(n)$
(resp.~$\ncd(n)$) the set of noncrossing partitions of type $B_n$ (resp.~type
$D_n$). Like type $A$, we have $\ncb(n)\cong\nca(B_n)$ and
$\ncd(n)\cong\nca(D_n)$. We note that $\ncb(n)$ and $\ncd(n)$ can also be
defined using circular representation, see \cite{Athanasiadis2005, Kim,
  Reiner1997}. However, the standard representation is more suitable for our
purpose.

A \emph{nonnesting partition of type $B_n$} is a partition $\pi\in\Pi_B(n)$ such
that $\pi_0$ is nonnesting with respect to the order $1\prec \cdots\prec n\prec
0\prec -n\prec \cdots\prec -1$, where $\pi_0$ is the partition of $[\pm
n]\cup\{0\}$ obtained from $\pi$ by adding $0$ to the zero block if $\pi$ has a
zero block, and by adding the singleton $\{0\}$ otherwise. See
Figure~\ref{fig:typeb_nonnesting} for an example.  A \emph{nonnesting partition
  of type $C_n$} is a partition $\pi\in\Pi_B(n)$ which is nonnesting with
respect to the order $1\prec \cdots\prec n\prec -n\prec \cdots\prec -1$. See
Figure~\ref{fig:typec_nonnesting} for an example.  A \emph{nonnesting partition
  of type $D_n$} is a partition $\pi\in\Pi_B(n)$ such that
\begin{enumerate}
\item if $\pi$ has a zero block $B$, then $\{n,-n\}\subsetneq B$,
\item $\pi'\in\nnb(n-1)$, where $\pi'$ is the partition obtained from $\pi$ by
  taking the union of the blocks containing $n$ or $-n$ and removing $n$ and $-n$.
\end{enumerate}
See Figure~\ref{fig:nonnest_D} for an example.  We denote by $\nnb(n)$
(resp.~$\nnc(n)$ and $\nnd(n)$) the set of nonnesting partitions of type $B_n$
(resp.~type $C_n$ and type $D_n$). Then we have $\nnb(n)\cong\nna(B_n)$,
$\nnc(n)\cong\nna(C_n)$ and $\nnd(n)\cong\nna(D_n)$.

\section{Partitions of classical types}
\label{sec:part-class-types}

For $\pi\in\Pi_B(n)$ and a block $B$ of $\pi$, let $B^+$ (resp.~$B^-$) denote
the set of positive (resp.~negative) integers in $B$. Note that $(-B)^+=-(B^-)$.
We define $\alpha(\pi), \beta(\pi)$ and $\gamma(\pi)$ as follows:
\begin{itemize}
\item $\alpha(\pi)$ is the partition in $\Pi(n)$ such that $A\in\alpha(\pi)$ if and only if $A=B^+$ for some $B\in\pi$, 
\item $\beta(\pi)$ is the set of blocks $A\in\alpha(\pi)$ such that $\pi$ has a
  block containing $A$ and at least one negative integer, 
\item $\gamma(\pi)$ is the matching on $\beta(\pi)$ such that
  $\{A_1,A_2\}\in\gamma(\pi)$ if and only if $A_1\ne A_2$ and $A_1\cup
  (-A_2)$ is a block of $\pi$.
\end{itemize}

\begin{example}\label{ex:ff}
  If $\pi=\{ \pm \{1,-3,6\}, \{2,4,-2,-4\}, \pm\{5,8\}, \pm\{7\} \}$, we have
  $\alpha(\pi)=\{\{1,6\}$, $\{2,4\}$, $\{3\}$, $\{5,8\}$, $\{7\}\}$,
  $\beta(\pi)=\{\{1,6\},\{2,4\},\{3\}\}$ and $\gamma(\pi)$ is the matching on
  $\beta(\pi)$ with the only one matching pair $ \{ \{1,6\}, \{3\}\}$.
\end{example}

Assume that a block $A\in\beta(\pi)$ is not matched in $\gamma(\pi)$. If $B$ is
the block of $\pi$ with $A=B^+$, we have $B^+\cap (-(B^-))\ne \emptyset$ because
otherwise $A$ would be matched with another block $A'=(-B)^+=-(B^-)$. Thus we
have an integer $i$ both in $B^+$ and $-(B^-)$, which implies $ i,-i\in
B$. Therefore $B$ is a zero block of $\pi$, which is unique. This argument shows
that $\gamma(\pi)$ is a maximal matching on $\beta(\pi)$. In other words, if
$|\beta(\pi)|$ is even, then $\gamma(\pi)$ is a complete matching on
$\beta(\pi)$; and if $|\beta(\pi)|$ is odd, then there is a unique unmatched
block $A\in\beta(\pi)$ in $\gamma(\pi)$, and in this case, $\pi$ has the zero
block $A\cup(-A)$.

It is easy to see that $\pi$ can be reconstructed from
$(\alpha(\pi),\beta(\pi),\gamma(\pi))$.  Thus we get the following proposition.
\begin{prop}\label{thm:typeb}
  The map $\pi\mapsto (\alpha(\pi),\beta(\pi),\gamma(\pi))$ is a bijection between $\Pi_B(n)$ and the set of triples $(\sigma, X,Y)$, where $\sigma\in\Pi(n)$, $X$ is a set of blocks of $\sigma$ and $Y$ is a maximal matching on $X$.
\end{prop}

Now we define $\alpha_0(\pi)=\alpha(\pi)\cup \{\{0\}\}$, which is a partition of
$[n]\cup\{0\}$, and $\gamma_0(\pi)$ to be the matching on the blocks of
$\alpha_0(\pi)$ defined as follows. If $\gamma(\pi)$ is a complete matching,
then the matching pairs of $\gamma(\pi)$ and $\gamma_0(\pi)$ are the same. If
there is an unmatched block $A$ in $\gamma(\pi)$, which is necessarily unique,
then the matching pairs of $\gamma_0(\pi)$ are those in $\gamma(\pi)$ and
$\{\{0\}, A\}$. Note that $\gamma_0(\pi)$ is not necessarily a maximal matching.

\begin{example}
  If $\pi$ is the partition in Example~\ref{ex:ff}, we have
\[
\alpha_0(\pi)=\{\{0\}, \{1,6\}, \{2,4\}, \{3\}, \{5,8\}, \{7\}\},
\]
and $\gamma_0(\pi)$ is the matching on $\alpha_0(\pi)$ with the two matching
pairs $\{ \{1,6\}, \{3\}\}$ and $\{\{0\},\{2,4\}\}$.
\end{example}
Since $\gamma_0(\pi)$ determines $\beta(\pi)$ and $\gamma(\pi)$, we get the following. 

\begin{prop}
The map $\pi\mapsto (\alpha(\pi),\gamma_0(\pi))$ is a bijection between $\Pi_B(n)$ and the set of pairs $(\sigma, X)$ where $\sigma\in\Pi(n)$ and $X$ is a matching on the blocks of the partition $\sigma\cup\{\{0\}\}$.
\end{prop}

If $\alpha(\pi)$ has $k$ blocks, then $\alpha_0(\pi)$ has $k+1$ blocks. Let
$A_1,A_2,\ldots,A_{k+1}$ be the blocks of $\alpha_0(\pi)$ with
$\max(A_1)<\max(A_2)<\cdots<\max(A_{k+1})$. By identifying the block $A_i$ with
the integer $i$, we can consider $\gamma_0(\pi)$ as a matching on $[k+1]$ or an
involution on $[k+1]$. Thus we get the following formula for the cardinality of
$\Pi_B(n)$.

\begin{cor}\label{cor:pib}
The cardinality of $\Pi_B(n)$ is equal to $$\sum_{k=1}^n S(n,k) t_{k+1},$$ where $S(n,k)$ is the Stirling number of the second kind and $t_n$ is the number of involutions on $[n]$.
\end{cor}

Note that the formula in Corollary~\ref{cor:pib} is a type $B$ analog of
$\#\Pi(n)=\sum_{k=1}^n S(n,k)$.

\section{Interpretations for noncrossing and nonnesting partitions}
\label{sec:interpr-type-b}

The following terminologies will be used for the rest of this paper.

An \emph{integer partition} $\lambda=(\lambda_1,\lambda_2,\ldots,\lambda_\ell)$
is a weakly decreasing sequence of positive integers. Each $\lambda_i$ is called
\emph{part} of $\lambda$ and $\ell$ is called \emph{length} of $\lambda$. We
define $|\lambda|$ to be the sum $\lambda_1+\lambda_2+\cdots+\lambda_\ell$ of
all parts of $\lambda$. We will also consider $\lambda$ as the multiset
$\{1^{m_1}, 2^{m_2},\ldots\}$, where $m_i$ is the number of parts equal to $i$
in $\lambda$.

For two multisets $A$ and $B$, let $A\doublecup B$ denote the multiset union of
$A$ and $B$. 

For a subset $S$ of $[n]$ and a partition $\pi$ of $S$, the \emph{type}
$\type(\pi)$ of $\pi$ is the integer partition $\lambda=\{1^{m_1},
2^{m_2},\ldots\}$ such that $m_i$ is equal to the number of blocks of size $i$
in $\pi$. The \emph{type} $\type(\pi)$ of a partition $\pi\in\Pi_B(n)$ is the
integer partition $\lambda=\{1^{m_1}, 2^{m_2},\ldots\}$ such that $m_i$ is equal
to the number of unordered pairs $(B,-B)$ of nonzero blocks of size $i$ in
$\pi$.

Recall the sets $\ncnn(n)$, $\ncna(n)$, $\nnna(n)$, $\ncnnd(n)$, $\ncnad(n)$ and
$\nnnad(n)$ in Definition~\ref{def1}.

\begin{notation}
From now on, if we write $\{A_1,A_2,\ldots,A_k\}_<$, it is automatically assumed
that $A_i$'s are sorted in increasing order by their largest elements, that is,
$\max(A_1)<\max(A_2)<\cdots<\max(A_k)$.
\end{notation}

For a set $X=\{A_1,A_2,\ldots,A_{2k}\}_<$ of even number of blocks, we define
$pairing(X)$ to be the following multiset:
$$pairing(X)=\{|A_1\cup A_{2k}|, |A_2\cup A_{2k-1}|, \ldots, 
|A_{k}\cup A_{k+1}|\}.$$

\subsection{Noncrossing partitions}

Let $\pi\in\ncb(n)$ and consider the map
$\pi\mapsto(\alpha(\pi),\beta(\pi),\gamma(\pi))$ in the previous section.  Since
$\pi$ is noncrossing with respect to the order $1\prec2\prec\cdots\prec n\prec
-1\prec -2\prec\cdots\prec -n$, one can easily see that $\alpha(\pi)\in\nca(n)$,
all the blocks in $\beta(\pi)$ are nonnested, and the matching $\gamma(\pi)$ is
uniquely determined by $\beta(\pi)$. For instance, if
$\beta(\pi)=\{A_1,A_2,\ldots,A_k\}_<$, then $\gamma(\pi)$ is the matching
consisting of $\{A_i, A_{k+1-i}\}$ for all $1\leq i\leq \lfloor k/2\rfloor$.

For $\pi\in\ncb(n)$, we define $\ncmap_B(\pi) = (\alpha(\pi),\beta(\pi))$. 
In other words, $\ncmap_B(\pi)$ is the pair $(\sigma,X)$ where $\sigma$ is the partition obtained from $\pi$ by removing all the negative integers and $X$ is the set of blocks of $\sigma$ which are properly contained in some blocks of $\pi$. Note that we have $\ncmap_B(\pi)\in\ncnn(n)$.

\begin{figure}
  \centering
  \begin{pspicture}(0,0)(21,3) 
\tsframe (0,0)(21,3) 
    \vvput{1}{1}1\vvput{2}{2}2\vvput{3}{3}3\vvput{4}{4}4\vvput{5}{5}5
    \vvput{6}{6}6\vvput{7}{7}7\vvput{8}{8}8\vvput{9}{9}9\vvput{10}{10}{10}
    \vvput{11}{m1}{-1}\vvput{12}{m2}{-2}\vvput{13}{m3}{-3}\vvput{14}{m4}{-4}
    \vvput{15}{m5}{-5}
    \vvput{16}{m6}{-6}\vvput{17}{m7}{-7}\vvput{18}{m8}{-8}\vvput{19}{m9}{-9}
    \vvput{20}{m10}{-10} \edge14 \edge45 \edge23 \edge79 \edge{m1}{m4}
    \edge{m4}{m5} \edge{m2}{m3} \edge{m7}{m9} \edge{10}{m1} \edge{9}{m7}
    \EDGE{5}{m10}
\end{pspicture}
\caption{The standard representation of an element in $\ncb(10)$ with respect to the order $1\prec2\prec\cdots\prec 10 \prec -1\prec-2\prec\cdots\prec -10$.}
  \label{fig:typeb_inter}
\end{figure}

\begin{example}
If $\pi\in\ncb(10)$ is the partition in Figure~\ref{fig:typeb_inter}, we have
  $\ncmap_B(\pi)=(\sigma,X)$, where 
\[
\sigma = \{\{1,4,5\}, \{2,3\}, \{6\}, \{7,9\}, \{8\}, \{10\}\}
\]
and $X=\{\{1,4,5\},\{7,9\}, \{10\}\}$.
\end{example}

From the construction, one can easily prove the following proposition.

\begin{prop}
The map $\ncmap_B:\ncb(n)\to\ncnn(n)$ is a bijection. Moreover, if $\ncmap_B(\pi)=(\sigma,X)$ and $X=\{A_1,A_2,\ldots,A_k\}_<$, then 
$$\type(\pi) = \type(\sigma\setminus X) \doublecup T,$$
where
$$T = \left\{
  \begin{array}{ll}
pairing(X),
& \mbox{if $k$ is even,}\\
pairing(X\setminus\{A_{(k+1)/2}\}),  
& \mbox{if $k$ is odd.}\\
\end{array}\right.
$$
\end{prop}

\begin{figure}
 \centering
  \begin{pspicture}(0,0)(21,3.2) 
\tsframe (0,0)(21,3.2) 
    \vvput{1}{1}1\vvput{2}{2}2\vvput{3}{3}3\vvput{4}{4}4\vvput{5}{5}5
    \vvput{6}{6}6\vvput{7}{7}7\vvput{8}{8}8\vvput{9}{9}9\vvput{10}{10}{10}
    \vvput{11}{m10}{-10}\vvput{12}{m1}{-1}\vvput{13}{m2}{-2}\vvput{14}{m3}{-3}
    \vvput{15}{m4}{-4}
    \vvput{16}{m5}{-5}\vvput{17}{m6}{-6}\vvput{18}{m7}{-7}\vvput{19}{m8}{-8}
    \vvput{20}{m9}{-9}
  \edge{1}{2} \edge{3}{5} \edge{6}{7}
  \edge{m1}{m2} \edge{m3}{m5} \edge{m6}{m7}
  \edge{8}{m1} \edge{5}{m10} \edge{10}{m3} \edge{m10}{m6} \edge{7}{10}
 \EDGE{2}{m8}
\end{pspicture}
\caption{The standard representation of an element in $\ncd(10)$ with respect to the order $1\prec2\prec\cdots \prec 10 \prec -10 \prec -1\prec-2\prec\cdots\prec -9$. Note that the locations of $10$ and $-10$ are not important.}
  \label{fig:typeD_inter2}
\end{figure}

Now we consider $\pi\in\ncd(n)$. Let $\pi'$ be the partition obtained from $\pi$
by taking the union of the blocks containing $n$ or $-n$ and removing $n$ and
$-n$.  Note that $\pi$ is uniquely determined by $\pi'$ and the block of $\pi$
containing $n$.  We define $\ncmap_D(\pi)=(\sigma,X,\epsilon)$, where $\sigma$,
$X$ and $\epsilon$ are obtained as follows.

\begin{enumerate}
\item If $\pi$ has the blocks $\pm\{n\}$ or $\pi$ has a zero block, then
  $(\sigma,X)=\ncmap_B(\pi')$ and $\epsilon=0$.
\item Otherwise, the block of $\pi$ containing $n$ can be written as $$\{a_1,a_2,\ldots,a_r,-b_1,-b_2,\ldots,-b_s,n\}$$
for some integers $r$, $s$, $a_1,\ldots,a_r$, $b_1,\ldots,b_s$ with $r,s\geq0$, $r+s\geq1$, $1\leq a_1<\cdots<a_r<n$ and $1\leq b_1<\cdots<b_s<n$. Let $\epsilon=1$ if $s=0$, or $r,s>0$ and $a_r < b_s$; and $\epsilon=-1$ otherwise. Let $\sigma$ be the partition of $[n-1]$ such that $A\in\sigma$ if and only if $A=B^+\setminus\{n\}$ for some $B\in\pi$ with $B^+\ne\emptyset$. Let $X$ be the set of blocks of $\sigma$ which are properly contained in some blocks of $\pi$.
\end{enumerate}
Note that $\ncmap_D(\pi)\in\ncnnd(n-1)$.

\begin{example}
  Let $\pi=\{\pm\{1,2,-8\}, \pm \{-3,-5,6,7,10\}, \pm\{4\}, \pm\{9\}\}$ as shown in Figure~\ref{fig:typeD_inter2}. Then $\ncmap_D(\pi)=(\sigma,X,\epsilon)$ where $\sigma=\{\{1,2\}, \{3,5\}, \{4\}, \{6,7\}, \{8\}, \{9\}\}$, $X=\{\{1,2\}$, $\{3,5\}$, $\{6,7\}$, $\{8\}\}$ and $\epsilon=-1$.
\end{example}

\begin{prop}\label{thm:3}
  The map $\ncmap_D:\ncd(n)\to\ncnnd(n-1)$ is a bijection. Moreover, if
  $\ncmap_D(\pi)=(\sigma,X,\epsilon)$ and $X=\{A_1,A_2,\ldots,A_k\}_<$, then
  $\type(\pi)= \type(\sigma\setminus X) \doublecup T$, where
$$T = \left\{
  \begin{array}{ll}
pairing(X) \doublecup \{1\}, 
& \mbox{if $\epsilon=0$ and $k=2t$,}\\
pairing(X\setminus\{A_{t+1}\}),  
& \mbox{if $\epsilon=0$ and $k=2t+1$,}\\
pairing(X\setminus\{A_{t},A_{t+1}\})
\doublecup \{|A_{t}| + |A_{t+1}| +1\}, 
& \mbox{if $\epsilon\ne0$ and $k=2t$,}\\
pairing(X\setminus\{A_{t+1}\}) \doublecup \{|A_{t+1}|+1\}, 
& \mbox{if $\epsilon\ne0$ and $k=2t+1$,}\\
 \end{array}\right.
$$
\end{prop}
\begin{proof}
  We will find the inverse map of $\ncmap_D$. Let
  $(\sigma,X,\epsilon)\in\ncnad(n-1)$ and
  $\pi'=(\ncmap_B)^{-1}(\sigma,X)\in\ncb(n-1)$.

If $\epsilon=0$, then $\pi\in\ncd(n)$ is the partition obtained from $\pi'$ by
adding $n$ and $-n$ to the zero block if $\pi'$ has a zero block; and by adding
the two singletons $\pm\{n\}$ otherwise.

Now assume $\epsilon\ne0$. If $k=2t$, then $\pi'$ has the blocks $\pm(A_t\cup
(-A_{t+1}))$. Then $\pi$ is the partition obtained from $\pi'$ by replacing
$\pm(A_t\cup (-A_{t+1}))$ with $\pm(\epsilon(A_t\cup(-A_{t+1}))\cup\{n\})$. Here
for a set $B$, the notation $\epsilon B$ means the set $\{\epsilon \cdot x:x\in B\}$.
If $k=2t+1$, then $\pi'$ has the blocks $\pm A_{t+1}$. Then $\pi$ is the partition obtained from $\pi'$ by replacing $\pm A_{t+1}$ with $\pm(\epsilon(A_{t+1})\cup\{n\})$. 

One can easily check that this is the inverse map of $\ncmap_D$. The `moreover' statement is obvious from the construction of the inverse map.
\end{proof}

\subsection{Nonnesting partitions}

As we did for noncrossing partitions, we can find interpretations for nonnesting partitions of classical types.

Consider the map $\pi\mapsto(\alpha(\pi),\beta(\pi),\gamma(\pi))$ for $\pi\in\nnb(n)$. 
It is easy to see that $\alpha(\pi)\in\nna(n)$, all the blocks in $\beta(\pi)$ are nonaligned and $\gamma(\pi)$ is determined from $\beta(\pi)$ as follows.
 Let $\beta(\pi)=\{A_1,A_2,\ldots,A_{2k}\}_<$ if $\beta(\pi)$ has even number of blocks; and
$\beta(\pi)=\{A_0,A_1,A_2,\ldots,A_{2k}\}_<$ otherwise. Then $\gamma(\pi)$ is the matching consisting of $\{A_i, A_{2k+1-i}\}$ for $i\in [k]$.

For $\pi\in\nnb(n)$, we define $\nnmap_B(\pi)=(\alpha(\pi),\beta(\pi))$.  In
other words, $\nnmap_B(\pi)$ is the pair $(\sigma,X)$ where $\sigma$ is the
partition obtained from $\pi$ by removing all the negative integers and $X$ is
the set of blocks of $\sigma$ which are properly contained in some blocks of
$\pi$. Note that we have $\nnmap_B(\pi)\in\nnna(n)$.

\begin{figure}
  \centering
%   \begin{pspicture}(0,0)(22,2) 
%     \vvput{1}{1}1\vvput{2}{2}2\vvput{3}{3}3\vvput{4}{4}4\vvput{5}{5}5
%     \vvput{6}{6}6\vvput{7}{7}7\vvput{8}{8}8\vvput{9}{9}9\vvput{10}{10}{10}
%     \vvput{11}{0}{0}\vvput{12}{m10}{-10}\vvput{13}{m9}{-9}\vvput{14}{m8}{-8}
%     \vvput{15}{m7}{-7}
%     \vvput{16}{m6}{-6}\vvput{17}{m5}{-5}\vvput{18}{m4}{-4}\vvput{19}{m3}{-3}
%     \vvput{20}{m2}{-2}\vvput{21}{m1}{-1}
%     \edge{1}{3} \edge{3}{7} \edge{2}{4} \edge{5}{9} \edge{6}{10}
%     \edge {m3}{m1} \edge {m7}{m3} \edge {m4}{m2} \edge {m9}{m5} \edge {m10}{m6}
%      \edge{7}{0} \edge{0}{m7} \edge{9}{m10} \edge{10}{m9}
% \end{pspicture}
\begin{pspicture}(0,0)(22,3.2)
\tsframe (0,0)(22,3.2)
  \vvput{1}{1}1\vvput{2}{2}2\vvput{3}{3}3\vvput{4}{4}4\vvput{5}{5}5
  \vvput{6}{6}6\vvput{7}{7}7\vvput{8}{8}8\vvput{9}{9}9\vvput{10}{10}{10}
  \vvput{11}{0}{0}\vvput{12}{m10}{-10}\vvput{13}{m9}{-9}\vvput{14}{m8}{-8}
  \vvput{15}{m7}{-7}
  \vvput{16}{m6}{-6}\vvput{17}{m5}{-5}\vvput{18}{m4}{-4}\vvput{19}{m3}{-3}
  \vvput{20}{m2}{-2}\vvput{21}{m1}{-1} \hcedge{1}{1} \hcedge{2}{1} \hcedge{3}{2}
  \hcedge{5}{2} \hcedge{6}{2} \hcedge{7}{2} \hcedge{9}{1.5} \hcedge{10}{1.5}
  \hcedge{11}{2} \hcedge{12}{2} \hcedge{13}{2} \hcedge{15}{2} \hcedge{18}{1}
  \hcedge{19}{1}
\end{pspicture}
 \caption{The standard representation of $\pi_0$ for a $\pi\in\nnb(10)$ with respect to the order $1\prec2\prec\cdots \prec 10 \prec 0\prec -10 \prec -9\prec\cdots\prec -1$.}
  \label{fig:typeb_nonnesting}
\end{figure}

\begin{example}
  Let $\pi=\{\{1,3,7,-7,-3,-1\}, \pm\{2,4\}, \pm\{5,9,-10,-6\}, \pm\{8\}\}\in\nnb(10)$ as shown in Figure~\ref{fig:typeb_nonnesting}. Then $\nnmap_B(\pi)=(\sigma,X)$ where $\sigma=\{\{1,3,7\}$, $\{2,4\}$, $\{5,9\}$, $\{6,10\}$, $\{8\}\}$ and $X=\{\{1,3,7\}, \{5,9\},\{6,10\}\}$.
\end{example}

From the construction, one can easily prove the following proposition.

\begin{prop}
The map $\nnmap_B:\nnb(n)\to\nnna(n)$ is a bijection. Moreover, if $\nnmap_B(\pi)=(\sigma,X)$ and $X=\{A_1,A_2,\ldots,A_k\}_<$, then 
$$\type(\pi) = \type(\sigma\setminus X) \doublecup T,$$
where
$$T = \left\{
  \begin{array}{ll}
pairing(X),
& \mbox{if $k$ is even,}\\
pairing(X\setminus\{A_{1}\}),  
& \mbox{if $k$ is odd.}\\
\end{array}\right.
$$
\end{prop}

Similarly, we define $\nnmap_C(\pi)=(\alpha(\pi),\beta(\pi))$ for
$\pi\in\nnc(n)$. Then we have $\nnmap_C(\pi)\in\nnna(n)$. Note that if
$\pi\in\nnc(n)$ and $\beta(\pi)=\{A_1,A_2,\ldots,A_k\}_<$, then $\gamma(\pi)$ is
the matching consisting of $\{A_i,A_{k+1-i}\}$ for all $i=1,2,\ldots,\floor{k/2}$.

\begin{example}
  Let $\pi=\{\pm\{1,3,7,-10,-6\}, \pm\{2,4\}, \{5,9,-9,-5\},
  \pm\{8\}\}\in\nnc(10)$ as shown in Figure~\ref{fig:typec_nonnesting}. Then
  $\nnmap_C(\pi)=(\sigma,X)$ where $\sigma=\{\{1,3,7\}$, $\{2,4\}$, $\{5,9\}$,
  $\{6,10\}$, $\{8\}\}$ and $X=\{\{1,3,7\}, \{5,9\},\{6,10\}\}$.
\end{example}

\begin{figure}
  \centering
%   \begin{pspicture}(0,0)(21,2) 
%     \vvput{1}{1}1\vvput{2}{2}2\vvput{3}{3}3\vvput{4}{4}4\vvput{5}{5}5
%     \vvput{6}{6}6\vvput{7}{7}7\vvput{8}{8}8\vvput{9}{9}9\vvput{10}{10}{10}
%     \vvput{11}{m10}{-10}\vvput{12}{m9}{-9}\vvput{13}{m8}{-8}
%     \vvput{14}{m7}{-7}
%     \vvput{15}{m6}{-6}\vvput{16}{m5}{-5}\vvput{17}{m4}{-4}\vvput{18}{m3}{-3}
%     \vvput{19}{m2}{-2}\vvput{20}{m1}{-1}
%     \edge{1}{3} \edge{3}{7} \edge{2}{4} \edge{5}{9} \edge{6}{10}
%     \edge {m3}{m1} \edge {m7}{m3} \edge {m4}{m2} \edge {m9}{m5} \edge {m10}{m6}
%      \edge{7}{m10} \edge{9}{m9} \edge{10}{m7}
% \end{pspicture}
\begin{pspicture}(0,0)(22,3.2) 
\tsframe (0,0)(22,3.2)
  \vvput{1}{1}1\vvput{2}{2}2\vvput{3}{3}3\vvput{4}{4}4\vvput{5}{5}5
  \vvput{6}{6}6\vvput{7}{7}7\vvput{8}{8}8\vvput{9}{9}9\vvput{10}{10}{10}
  \vvput{11}{m10}{-10}\vvput{12}{m9}{-9}\vvput{13}{m8}{-8} \vvput{14}{m7}{-7}
  \vvput{15}{m6}{-6}\vvput{16}{m5}{-5}\vvput{17}{m4}{-4}\vvput{18}{m3}{-3}
  \vvput{19}{m2}{-2}\vvput{20}{m1}{-1} \hcedge{1}{1} \hcedge{2}{1} \hcedge{3}{2}
  \hcedge{5}{2} \hcedge{6}{2} \hcedge{7}{2} \hcedge{9}{1.5} \hcedge{10}{2}
  \hcedge{11}{2} \hcedge{12}{2} \hcedge{14}{2} \hcedge{17}{1} \hcedge{18}{1}
\end{pspicture}
 \caption{The standard representation of an element in $\nnc(10)$ with respect to the order $1\prec2\prec\cdots \prec 10 \prec -10 \prec -9\prec\cdots\prec -1$.}
  \label{fig:typec_nonnesting}
\end{figure}

Then we get the following proposition in the same way.

\begin{prop}
  The map $\nnmap_C:\nnc(n)\to\nnna(n)$ is a bijection. Moreover, if
  $\nnmap_C(\pi)=(\sigma,X)$ and $X=\{A_1,A_2,\ldots,A_k\}_<$, then
$$\type(\pi) = \type(\sigma\setminus X) \doublecup T,$$
where
$$T = \left\{
  \begin{array}{ll}
pairing(X),
& \mbox{if $k$ is even,}\\
pairing(X\setminus\{A_{(k+1)/2}\}),  
 & \mbox{if $k$ is odd.}\\
\end{array}\right.
$$
\end{prop}

Now we consider nonnesting partitions of type $D_n$.  Let $\pi\in\nnd(n)$ and
let $\pi'$ be the partition obtained from $\pi$ by unioning the blocks
containing $n$ or $-n$ and removing $n$ and $-n$. Then $\nnmap_D(\pi)$ is
defined in the same way as $\ncmap_D(\pi)$. That is, we define
$\nnmap_D(\pi)=(\sigma,X,\epsilon)$, where $\sigma$ and $X$ are constructed as
follows.
\begin{enumerate}
\item If $\pi$ has the blocks $\pm\{n\}$ or $\pi$ has a zero block, then
  $(\sigma,X)=\nnmap_B(\pi')$ and $\epsilon=0$.
\item Otherwise, the block of $\pi$ containing $n$ can be written as $$\{a_1,a_2,\ldots,a_r,-b_1,-b_2,\ldots,-b_s,n\}$$
for some integers $r$, $s$, $a_1,\ldots,a_r$, $b_1,\ldots,b_s$ with $r,s\geq0$, $r+s\geq1$, $1\leq a_1<\cdots<a_r<n$ and $1\leq b_1<\cdots<b_s<n$. Let $\epsilon=1$ if $s=0$ or $r,s>0$ and $a_r < b_s$; and $\epsilon=-1$ otherwise. Let $\sigma$ be the partition of $[n-1]$ such that $A\in\sigma$ if and only if $A=B^+\setminus\{n\}$ for some $B\in\pi$ with $B^+\ne\emptyset$. Let $X$ be the set of blocks of $\sigma$ which are properly contained in some blocks of $\pi$.
\end{enumerate}
Note that $\nnmap_D(\pi)\in\nnnad(n-1)$.

\begin{figure}
  \centering
%   \begin{pspicture}(0,0)(21,2) 
%     \vvput{1}{1}1\vvput{2}{2}2\vvput{3}{3}3\vvput{4}{4}4\vvput{5}{5}5
%     \vvput{6}{6}6\vvput{7}{7}7\vvput{8}{8}8\vvput{9}{9}9\vvput{10}{10}{10}
%     \vvput{11}{m10}{-10}\vvput{12}{m9}{-9}\vvput{13}{m8}{-8}
%     \vvput{14}{m7}{-7}
%     \vvput{15}{m6}{-6}\vvput{16}{m5}{-5}\vvput{17}{m4}{-4}\vvput{18}{m3}{-3}
%     \vvput{19}{m2}{-2}\vvput{20}{m1}{-1}
%     \edge{1}{4} \edge{4}{7} \edge{3}{6} \edge{5}{9} 
%     \edge{m4}{m1} \edge{m7}{m4} \edge{m6}{m3} \edge{m9}{m5} 
%     \edge{6}{m10} \edge{m10}{m7} \edge{8}{m9} \edge{9}{m8} 
%      \edge{7}{10} \edge{10}{m6} 
% \end{pspicture}
\begin{pspicture}(0,0)(21,3.7) 
\tsframe (0,0)(21,3.7) 
  \vvput{1}{1}1\vvput{2}{2}2\vvput{3}{3}3\vvput{4}{4}4\vvput{5}{5}5
  \vvput{6}{6}6\vvput{7}{7}7\vvput{8}{8}8\vvput{9}{9}9\vvput{10}{10}{10}
  \vvput{11}{m10}{-10}\vvput{12}{m9}{-9}\vvput{13}{m8}{-8} \vvput{14}{m7}{-7}
  \vvput{15}{m6}{-6}\vvput{16}{m5}{-5}\vvput{17}{m4}{-4}\vvput{18}{m3}{-3}
  \vvput{19}{m2}{-2}\vvput{20}{m1}{-1} \hcedge{1}{1.5} \hcedge{3}{1.5}
  \hcedge{5}{2} \hcedge{6}{2.5} \hcedge{7}{1.5} \hcedge{8}{2} \hcedge{9}{2}
  \hcedge{10}{2.5} \hcedge{11}{1.5} \hcedge{12}{2} \hcedge{14}{1.5}
  \hcedge{15}{1.5} \hcedge{17}{1.5}
\end{pspicture}
\caption{The standard representation of an element in $\nnd(10)$ with respect to the order $1\prec2\prec\cdots \prec 10 \prec -10 \prec -9\prec\cdots\prec -1$. Note that the locations of $10$ and $-10$ are not important.}
  \label{fig:nonnest_D}
\end{figure}

\begin{example}
  Let $\pi=\{\pm\{1,4,7,-3,-6,10\}, \pm \{2\}, \pm\{5,9,-8\}\}\in\nnd(10)$ as
  shown in Figure~\ref{fig:nonnest_D}. Then $\nnmap_D(\pi)=(\sigma,X,\epsilon)$
  where $\sigma=\{\{1,4,7\}$, $\{2\}$, $\{3,6\}$, $\{5,9\}$, $\{8\}\}$,
  $X=\{\{3,6\}$, $\{1,4,7\}$, $\{8\}$, $\{5,9\}\}$ and $\epsilon=-1$.
\end{example}

\begin{prop}
The map $\nnmap_D:\nnd(n)\to\nnnad(n-1)$ is a bijection. Moreover, if $\nnmap_D(\pi)=(\sigma,X,\epsilon)$ and $X=\{A_1,A_2,\ldots,A_k\}_<$, then   
$\type(\pi)= \type(\sigma \setminus X) \doublecup T$, where
$$T = \left\{
  \begin{array}{ll}
pairing(X) \doublecup \{1\}, 
& \mbox{if $\epsilon=0$ and $k$ is even,}\\
pairing(X\setminus\{A_{1}\}),  
& \mbox{if $\epsilon=0$ and $k$ is odd,}\\
pairing(X\setminus\{A_{1},A_{2}\})
\doublecup \{|A_{1}| + |A_{2}| +1\}, 
& \mbox{if $\epsilon\ne0$ and $k$ is even,}\\
pairing(X\setminus\{A_{1}\}) \doublecup \{|A_{1}|+1\}, 
& \mbox{if $\epsilon\ne0$ and $k$ is odd.}\\
 \end{array}\right.
$$
\end{prop}
\begin{proof}
The proof is similar to that of Proposition~\ref{thm:3}, hence we omit it.  
\end{proof}

\section{Type-preserving Bijections}
\label{sec:type-preserving}

In the previous section we have interpreted noncrossing and nonnesting
partitions of types $B_n$, $C_n$ and $D_n$ in terms of noncrossing and
nonnesting partitions of type $A_{n-1}$ or $A_{n-2}$. In this section we find
type-preserving bijections between noncrossing and nonnesting partitions of
types $B_n$, $C_n$ and $D_n$ using the following theorem as one of the building
blocks. 

\begin{thm}\label{thm:1}\cite[Theorem~3.1]{Athanasiadis1998}
  Suppose $\{A_1,A_2,\ldots,A_k\}_<$ is the set of blocks of
  $\sigma\in\nca(n)$. Then there is a unique element $\sigma'\in\nna(n)$ such
  that $\{A'_1,A'_2,\ldots,A'_k\}_<$ is the set of blocks of $\sigma'$ with
  $\max(A_i)=\max(A'_i)$ and $|A_i|=|A'_i|$ for all $i\in[k]$.
\end{thm}

The above theorem follows from the observation that any partition in $\nca(n)$
or $\nna(n)$ is completely determined by the largest elements and the sizes of
the blocks. For example, the largest elements (circled vertices) and the sizes
(integers above vertices) of the blocks of the partition in Figure~\ref{fig:nca}
are represented below.

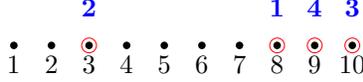
\begin{center}
\begin{pspicture}(0,0)(11,2.5)
\tsframe (0,0)(11,2.5)
\multido{\n=1+1}{10}{\vput{\n}}
\activevertex3[2]
\activevertex8[1]
\activevertex9[4]
\activevertex{10}[3]
\end{pspicture}
\end{center}

One can check that there are a unique noncrossing partition and a unique
nonnesting partition whose largest elements and sizes of the blocks can
represented as above. For instance, if it is a noncrossing partition, then $7$
must be connected to $9$ or $10$, where it cannot be connected to $10$ because
the arc $(7,10)$ and the arc $(i,9)$ for some $i<7$ will create a crossing. Thus
$7$ is connected to $9$. In this way we can uniquely determine all arcs from the
right. It is similar for a nonnesting partition. The unique nonnesting partition
for the above diagram is the partition in Figure~\ref{fig:nna}.

For $\sigma\in\nca(n)$, let $\rho(\sigma)$ be the unique element
$\sigma'\in\nna(n)$ in Theorem~\ref{thm:1}. For instance, if $\sigma$ is the
partition in Figure~\ref{fig:nca}, then $\rho(\sigma)$ is the one in
Figure~\ref{fig:nna}. It is clear from Theorem~\ref{thm:1} that the map
$\rho:\nca(n)\to\nna(n)$ is a type-preserving bijection, which also preserves
the largest elements of the blocks.  We can naturally extend the map $\rho$ to a
map from $\ncna(n)$ to $\nnna(n)$. In order to do this, we need the following
lemma.

\begin{lem}\label{thm:2}
  Suppose $\{A_1,\ldots,A_k\}_<$ and $\{A_1',A_2',\ldots,A_k'\}_<$ are the sets
  of blocks of $\sigma\in\nca(n)$ and $\rho(\sigma)\in\nna(n)$ respectively.
  Then $A_i$ is a nonaligned block of $\sigma$ if and only if $A_i'$ is a
  nonaligned block of $\rho(\sigma)$.
\end{lem}
\begin{proof}
  By definition, $A_i$ is aligned if and only if there is an integer $t$ such
  that $\max(A_i)<t$ and $t\ne\max(A_j)$ for all $j\in[k]$. Thus $A_{k-i}$ is
  nonaligned if and only if $\max(A_{k-i})=n-i$.  Since $\max(A_i)=\max(A_i')$
  for all $i\in[k]$, we are done.
\end{proof}

Now we define a map $\ovrho:\ncna(n)\to\nnna(n)$.  For $(\sigma,X)\in\ncna(n)$,
suppose that $\{A_1,A_2,\ldots,A_k\}_<$ is the set of blocks of $\sigma$ and
$X=\{A_{i_1},A_{i_2}, \ldots, A_{i_r}\}_<$. Suppose also that
$\{A'_1,A'_2,\ldots,A'_k\}_<$ is the set of blocks of $\sigma'=\rho(\sigma)$ and
$X'=\{A'_{i_1},A'_{i_2}, \ldots, A'_{i_r}\}_<$.  Then we define
$\ovrho(\sigma,X)=(\sigma',X')$. In other words, if we identify a block $A$ with
its largest element $a=\max(A)$, then $\ovrho (\sigma,(a_1,a_1,\dots,a_k)) =
(\rho(\sigma),(a_1,a_1,\dots,a_k))$.  For example, if $\sigma$ is the partition
in Figure~\ref{fig:nca} and $X=\{\{8\}, \{1,4,10\}\}$ then
$\ovrho(\sigma,X)=(\sigma',X')$, where $\sigma'$ is the partition in
Figure~\ref{fig:nna} and $X'=\{\{8\},\{5,7,10\}\}$. Note that the largest elements
of the blocks in $X$ are exactly those in $X'$.

By Lemma~\ref{thm:2}, we have
$\ovrho(\sigma,X)\in \nnna(n)$. Thus we get the following proposition.

\begin{prop}
  The map $\ovrho:\ncna(n)\to\nnna(n)$ is a bijection such that if
  $\ovrho(\sigma,X)=(\sigma',X')$ and $X=\{A_1,A_2,\ldots,A_k\}_<$, then
  $\type(\sigma)=\type(\sigma')$ and $X'=\{A_1',A_2',\ldots,A_k'\}_<$ with
  $\max(A_i)=\max(A_i')$ and $|A_i|=|A_i'|$ for all $i\in[k]$.
\end{prop}

\subsection{Interchanging nonnested blocks and nonaligned blocks}

In this subsection we will construct an involution on $\nca(n)$ which interchanges
nonnested blocks and nonaligned blocks. In order to do this we need several
definitions.

For $\pi\in\nca(n)$ and $S=\{a_1,a_2,\ldots,a_k\}$ with $1\leq
a_1<a_2<\cdots<a_k\leq n$, we define $\pi\cap S$ to be the partition of $[k]$
obtained from $\pi$ by removing all the integers not in $S$ and replacing $a_i$
with $i$ for each $i\in[k]$.

For two partitions $\sigma\in\nca(n)$ and $\tau\in\nca(m)$, we define
$\sigma\uplus\tau$ to be the partition in $\nca(n+m)$ obtained from $\sigma$ by
adding all the blocks of $\tau$ whose elements are increased by $n$.  Ignoring
the labels, the standard representation of $\sigma \uplus \tau$ looks as
follows:
\begin{center}
\psset{unit=6pt}
\begin{pspicture}(-6.5,0)(9,2) 
\tsframe (-6.5,0)(9,2)
\rput[r](0,1){$\sigma \uplus \tau$ = }
  \rect0{$\sigma$} \rect5{$\tau$}
\end{pspicture}
\end{center}

If $\pi\in\nca(n)$ cannot be expressed as $\pi=\sigma\uplus\tau$ for some
$\sigma\in\nca(r)$ and $\tau\in\nca(s)$ with $r,s\geq1$, then we say that $\pi$
is \emph{connected}. Since $\pi\in\nca(n)$ is a noncrossing partition, $\pi$ is
connected if and only if $1$ and $n$ are in the same block.

For a connected partition $\sigma\in\nca(n)$ and any partition $\tau\in\nca(m)$,
we define $\sigma * \tau$ to be the partition in $\nca(n+m+1)$ obtained from
$\sigma \uplus \tau$ by adding $n+m+1$ to the block containing $n$.  Thus the
standard representation of $\sigma \uplus \tau$ looks as follows (here a
half-circle means a connected partition and a round-rectangle means any
partition):
\begin{center}
\psset{unit=6pt}
\begin{pspicture}(-6.5,0)(10,3) 
\tsframe (-6.5,0)(10,3)
\rput[r](0,1){$\sigma * \tau$ = }
 \nonnested{0}{$\sigma$}{$\tau$}
\end{pspicture}
\end{center}
For example,
\begin{center}
\begin{pspicture}(0,0)(4,2)
\vvput11{} \vvput22{} \vvput33{} \vvput44{} 
\edge12 \edge24
\end{pspicture}
\begin{pspicture}(-1.5,0)(3,2)
\rput(-.5,1){\huge $*$}
\vvput11{} \vvput22{} \vvput33{} 
\edge12
\end{pspicture}
\begin{pspicture}(-1.5,0)(9,2)
\rput(-.5,1){\huge $=$}
\vvput11{} \vvput22{} \vvput33{} 
\vvput44{} \vvput55{} \vvput66{} \vvput77{} \vvput88{}
\edge12 \edge24 \edge56 \edge48
\end{pspicture}
\end{center}

We also consider $\sigma * \tau$ when one (or both) of $\sigma$ and $\tau$ is
the empty partition $\emptyset$: $\emptyset * \emptyset$ is the unique partition
$\{\{1\}\}$ in $\Pi(1)$, $\emptyset * \tau$ is $\tau \cup \{\{m+1\}\}$ and
$\sigma * \emptyset$ is the partition obtained from $\sigma$ by adding $n+1$ to
the block containing $n$.

For $\pi\in\nca(n)$, we define two maps $decomp_1(\pi)$ and $decomp_2(\pi)$ as
follows.  If $\{n\}$ is not a block of $\pi$, then we can uniquely
decompose $\pi$ as $\pi = \sigma \uplus(\tau *\upsilon)$, see the diagram
below. 
\begin{center}
\psset{unit=6pt}
\begin{pspicture}(-3.5,0)(15,3)
\tsframe(-3.5,0)(15,3)
\rput[r](0,1){$\pi$ = }
  \rect0{$\sigma$}  \nonnested{5}{$\tau$}{$\upsilon$}
\end{pspicture}
\end{center}
In this case, we define $decomp_1(\pi) = decomp_2(\pi) =
(\sigma,\tau,\upsilon)$. If $\{n\}$ is a block of $\pi$, then we define
$decomp_1(\pi) = (\pi\cap [n-1],\emptyset,\emptyset)$ and $decomp_2(\pi) =
(\emptyset,\emptyset, \pi\cap [n-1])$.  Note that if $decomp_1(\pi)
=(\sigma,\tau,\upsilon)$ or $decomp_2(\pi) =(\sigma,\tau,\upsilon)$, we always
have $\pi = \sigma \uplus(\tau *\upsilon)$.  Moreover, if $decomp_1(\pi)
=(\sigma,\tau,\upsilon)$ and $\tau=\emptyset$, then $\upsilon=\emptyset$,
whereas, if $decomp_2(\pi) =(\sigma,\tau,\upsilon)$ and $\tau=\emptyset$, then
$\sigma=\emptyset$.

Now we are ready to define a map $\xi:\nca(n)\to \nca(n)$.  First, we assume
that $\{n\}$ is not a block of $\pi\in \nca(n)$. Suppose also that
$\pi$ has $r$ nonnested blocks and $s$ nonaligned blocks.

For $i\in[r]$, let $decomp_1(\pi_i)=(\pi_{i+1}, \sigma_i, \sigma_i')$, where
$\pi_1=\pi$. Since $\pi$ has $r$ nonnested blocks, we have $\pi_i\ne\emptyset$
for $i\in[r]$ and $\pi_{r+1}=\emptyset$. Thus
\begin{align*}
  \pi = \pi_1 &= \pi_2 \uplus (\sigma_1 * \sigma_1')\\
  &= \pi_3 \uplus (\sigma_2 * \sigma_2') \uplus (\sigma_1 * \sigma_1')\\
  &  \qquad \qquad \vdots\\
  &= (\sigma_r * \sigma_r') \uplus (\sigma_{r-1} * \sigma_{r-1}') \uplus \cdots
  \uplus (\sigma_1 * \sigma_1').
\end{align*}
Pictorially, the above decomposition of $\pi$ can be represented as follows.
\begin{center}
\psset{unit=6pt}
 \begin{pspicture}(-5,0)(44,3)  
\tsframe (-5,0)(44,3) 
\rput[r](-2,1){$\pi=$}
\nonnested0{$\sigma_r$}{$\sigma_r'$}
\rput(16,1){$\cdots$}
\nonnested{22}{$\sigma_2$}{$\sigma_2'$}
\nonnested{33}{$\sigma_1$}{$\sigma_1$}
\end{pspicture}
\end{center}
Note that $\sigma_1\ne\emptyset$, and for $2\leq i\leq r$, if
$\sigma_i=\emptyset$, the $\sigma_i'=\emptyset$. If $\{N_1,N_2,\ldots,N_r\}_<$
is the set of all nonnested blocks of $\pi$, then $|N_i|-1$ is equal to the size
of the block of $\sigma_{r+1-i}$ containing the largest integer if
$\sigma_{r+1-i}\ne\emptyset$; and $0$ if $\sigma_{r+1-i}=\emptyset$.

Similarly, for $i\in[s]$, let $decomp_2(\upsilon_i) =
(\tau_i',\tau_i,\upsilon_{i+1})$, where $\upsilon_1=\pi$. Since $\pi$ has $s$
nonaligned blocks, we have $\upsilon_i\ne\emptyset$ for $i\in[s]$ and
$\upsilon_{s+1}=\emptyset$.  Thus
\begin{align*}
\pi =\upsilon_1 & = \tau_1' \uplus (\tau_1 * \upsilon_2)\\
& = \tau_1' \uplus (\tau_1 * (\tau_2' \uplus (\tau_2 *\upsilon_3)))\\
&  \qquad \qquad \vdots\\
& = \tau_1' \uplus (\tau_1 * (\tau_2' \uplus (\tau_2 * 
(\tau_3' \uplus \cdots (\tau_s'\uplus (\tau_s * \emptyset )) \cdots ).\\
\end{align*}
Pictorially, the above decomposition of $\pi$ can be represented as follows.
 \begin{center}
\psset{unit=6pt}
  \begin{pspicture}(6,0)(61,10)  
\tsframe (6,0)(61,10) 
\rput[r](10,1){$\pi=$}
\rput(35,1){$\cdots$}
\rect{12}{$\tau_1'$}
\halfcircle{17}{$\tau_1$}
\rect{22}{$\tau_2'$}
\halfcircle{27}{$\tau_2$}
\rect{37}{$\tau_{s-1}'$}
\halfcircle{42}{$\tau_{s-1}$}
\rect{47}{$\tau_s'$}
\halfcircle{52}{$\tau_s$}
\connect(56,1)
\pnode(21,0){26}
\pnode(61,0){61}
\pnode(60,0){60}
\pnode(59,0){59}
\pnode(31,0){36}
\pnode(46,0){46}
\ncarc[arcangle=80]{26}{61}
\ncarc[arcangle=80]{36}{60}
\ncarc[arcangle=80]{46}{59}
\end{pspicture}
\end{center}
Note that $\tau_1\ne\emptyset$, and for $2\leq i\leq s$, if $\tau_i=\emptyset$,
the $\tau_i'=\emptyset$. If $\{A_1,A_2,\ldots,A_s\}_<$ is the set of all
nonaligned blocks of $\pi$, then $|A_i|-1$ is equal to the size of the block of
$\tau_{s+1-i}$ containing the largest integer if $\tau_{s+1-i}\ne\emptyset$; and
$0$ if $\tau_{s+1-i}=\emptyset$.

Since $\{n\}$ is not a block of $\pi$, we have $decomp_1(\pi)= decomp_2(\pi)$,
thus $\pi_2=\tau_1'$, $\sigma_1=\tau_1$ and $\sigma_1'=\upsilon_2$. Thus we get
the following:
\[
\pi=(\sigma_r * \sigma_r') \uplus \cdots \uplus (\sigma_{2} * \sigma_{2}')
\uplus (\tau_1 * (\tau_2' \uplus (\tau_2 * (\tau_3' \uplus \cdots (\tau_s'\uplus
(\tau_s * \emptyset )) \cdots ).
\] 
Then we define
\[
\xi(\pi)=(\tau_s * \tau_s') \uplus \cdots \uplus (\tau_{2} * \tau_{2}') \uplus
(\sigma_1 * (\sigma_2' \uplus (\sigma_2 * (\sigma_3' \uplus \cdots
(\sigma_r'\uplus (\sigma_r * \emptyset )) \cdots ).
\]
See Figure~\ref{fig:xi}.

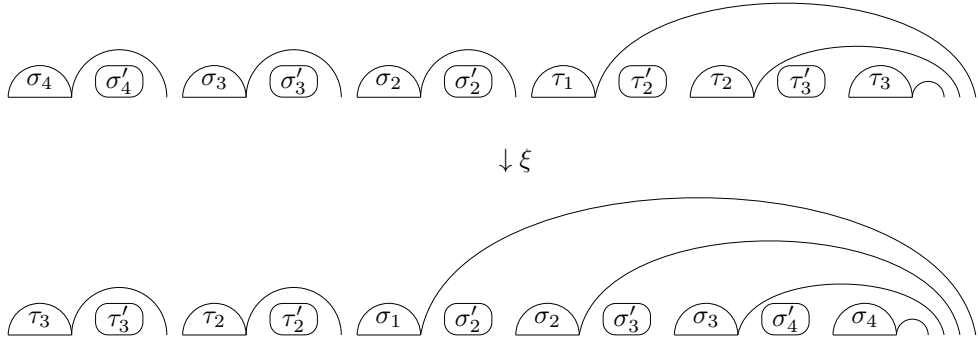
\begin{figure}
  \centering
\psset{unit=6pt}
  \begin{pspicture}(0,-15)(61,6) 
\tsframe (0,-15)(61,6) 
\nonnested0{$\sigma_4$}{$\sigma_4'$}
\nonnested{11}{$\sigma_3$}{$\sigma_3'$}
\nonnested{22}{$\sigma_2$}{$\sigma_2'$}
\halfcircle{33}{$\tau_1$}
\rect{38}{$\tau_2'$}
\halfcircle{43}{$\tau_2$}
\rect{48}{$\tau_3'$}
\halfcircle{53}{$\tau_3$}
\connect(57,1)
\pnode(37,0){37}
\pnode(47,0){47}
\pnode(60,0){60}
\pnode(61,0){61}
\ncarc[arcangle=80]{37}{61}
\ncarc[arcangle=80]{47}{60}
\rput(32,-4) {$\downarrow\xi$}
\rput(0,-15){
\nonnested0{$\tau_3$}{$\tau_3'$}
\nonnested{11}{$\tau_2$}{$\tau_2'$}
\halfcircle{22}{$\sigma_1$}
\rect{27}{$\sigma_2'$}
\halfcircle{32}{$\sigma_2$}
\rect{37}{$\sigma_3'$}
\halfcircle{42}{$\sigma_3$}
\rect{47}{$\sigma_4'$}
\halfcircle{52}{$\sigma_4$}
\connect(56,1)
\pnode(26,0){26}
\pnode(61,0){61}
\pnode(60,0){60}
\pnode(59,0){59}
\pnode(36,0){36}
\pnode(46,0){46}
\ncarc[arcangle=80]{26}{61}
\ncarc[arcangle=80]{36}{60}
\ncarc[arcangle=80]{46}{59}
}
\end{pspicture}
\caption{Illustration of the map $\xi$. We have $\sigma_1=\tau_1$.}
  \label{fig:xi}
\end{figure}

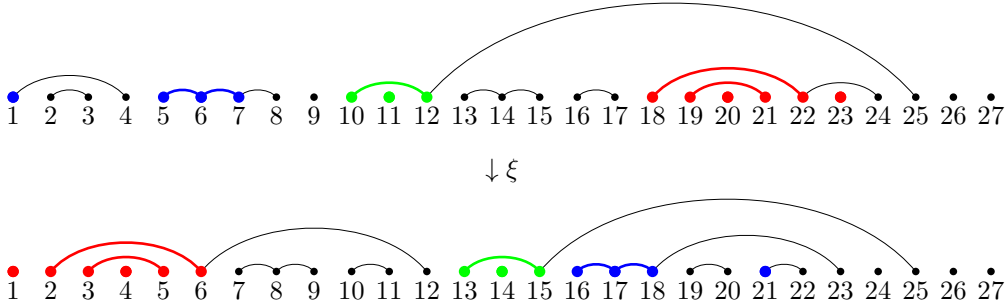
\begin{figure}
  \centering
\begin{pspicture}(0,0)(28,3.5)
\tsframe (0,0)(28,3.5)
 \multido{\n=1+1}{27}{\vput{\n}}
  \edge{18}{22} \edge{22}{24} \edge{19}{21}
   \edge{13}{14} \edge{14}{15} \edge{16}{17}
   \edge{10}{12} \edge{12}{25}
   \edge{5}{6} \edge{7}{8} \edge{6}{7}
\edge{2}{3}
   \edge{1}{4}
\tedge11{blue}
\tedge56{blue}
\tedge67{blue}
\tedge{10}{12}{green}
\tedge{11}{11}{green}
\tedge{18}{22}{red}
\tedge{19}{21}{red}
\tedge{20}{20}{red}
\tedge{23}{23}{red}
\end{pspicture}
$$\downarrow \xi$$
\begin{pspicture}(0,0)(28,3) 
\tsframe (0,0)(28,3)
 \multido{\n=1+1}{27}{\vput{\n}}
   \edge26 \edge35 \edge6{12}
    \edge78 \edge89 \edge{10}{11}
   \edge{13}{15} \edge{15}{25}
   \edge{16}{17} \edge{17}{18} \edge{18}{23}
  \edge{19}{20}
   \edge{21}{22}
\tedge{13}{15}{green}
\tedge{14}{14}{green}
\tedge{16}{17}{blue}
\tedge{17}{18}{blue}
\tedge{21}{21}{blue}
\tedge{1}{1}{red}
\tedge{2}{6}{red}
\tedge{3}{5}{red}
\tedge{4}{4}{red}
\end{pspicture}
\caption{An example of the map $\xi$. In the upper diagram, $\sigma_1=\tau_1$ is
  colored green, $\sigma_i$'s are colored blue $\tau_i$'s are colored red for
  $i\geq2$.}
  \label{fig:example_decomp}
\end{figure}

Now let $\pi$ be any element in $\nca(n)$. If $k$ is the largest integer such
that $k\leq n$ and $\{k\}$ is not a block of $\pi$, we define $\xi(\pi)$ to be
the partition obtained from $\xi(\pi\cap[k])$ by adding the blocks
$\{k+1\},\{k+2\},\ldots,\{n\}$. See Figure~\ref{fig:example_decomp}. 

For $\pi\in\nca(n)$, let $\nne(\pi)$ (resp.~$\nal(\pi)$) denote the number of
nonnested (resp.~nonaligned) blocks of $\pi$. From the construction of $\xi$, it
is easy to see that the following theorem holds.

\begin{thm}\label{thm:4}
  The map $\xi$ is a type-preserving involution on $\nca(n)$ satisfying
  $\nne(\xi(\pi))=\nal(\pi)$ and $\nal(\xi(\pi))=\nne(\pi)$.  Moreover, if
  $\{N_1,N_2,\ldots,N_r\}_<$, $\{N'_1,N'_2,\ldots,N'_s\}_<$,
  $\{A_1,A_2,\ldots,A_s\}_<$ and $\{A'_1,A'_2,\ldots,A'_r\}_<$ are the set of
  nonnested blocks of $\pi$ and $\xi(\pi)$ and the set of nonaligned blocks of
  $\pi$ and $\xi(\pi)$ respectively, then $|N_i|=|A'_i|$ and $|A_j|=|N'_j|$ for
  all $i\in[r]$ and $j\in[s]$.
\end{thm}

The following corollary is an immediate consequence of Theorem~\ref{thm:4}.

\begin{cor}\label{thm:gf1} We have
$$\sum_{\pi\in\nca(n)} x^{\nne(\pi)} y^{\nal(\pi)} = \sum_{\pi\in\nca(n)} x^{\nal(\pi)} y^{\nne(\pi)}.$$  
\end{cor}

In fact, we can find a formula for the following generating function:
$$  F(x,y,z) = \sum_{n\geq0} \left( \sum_{\pi\in\nca(n)} x^{\nne(\pi)} y^{\nal(\pi)} \right) z^n.$$

Let $\nca'(n)$ denote the set of connected partitions in $\nca(n)$. We define
\begin{align*}
  C(z) &= \sum_{n\geq0} \#\nca(n) z^n = \frac{1-\sqrt{1-4z}}{2z}, \qquad
  B(z) = \sum_{n\geq1} \#\nca'(n) z^n,\\
  A(x,z) &= \sum_{n\geq0} \left( \sum_{\pi\in\nca(n)} x^{\nne(\pi)} \right) z^n = \sum_{n\geq0} \left( \sum_{\pi\in\nca(n)} x^{\nal(\pi)} \right) z^n.\\
\end{align*}
It is not difficult to see that
\[
  C(z) = \frac1{1-B(z)}, \qquad  A(x,z) = \frac1{1-xB(z)}.
\]
Using the decomposition $\pi = \sigma \uplus(\tau *\upsilon)\uplus \mu$, where
$\mu$ is a partition consisting of singletons and $\tau$ is a connected
partition, one can also show that
\[ 
F(x,y,z) = \frac1{1-xyz} \left( 1 + xyz A(x,z) A(y,z) B(z) \right).
\]
Solving the above equations, we get the following generating function.

\begin{prop}\label{thm:gf}
We have
  $$F(x,y,z) = \frac1{1-xyz} \left( 
1 + \frac{2xyz(3+\sqrt{1-4z})}{(1-3x-x\sqrt{1-4z})(1-3y-y\sqrt{1-4z})}
\right).$$
\end{prop}

We can naturally extend $\xi$ to the map $\ovxi:\ncnn(n)\to\ncna(n)$ defined as
follows. Let $(\sigma,X)\in\ncnn(n)$ and $\sigma'=\xi(\sigma)$. Suppose
$\{A_1,A_2,\ldots,A_k\}_<$ is the set of all nonnested blocks of $\sigma$ and
$\{A'_1,A'_2,\ldots,A'_k\}_<$ is the set of all nonaligned blocks of
$\sigma'$. Then we can write $X=\{A_{i_1},A_{i_2},\ldots,A_{i_r}\}_<$.  We
define $\ovxi(\sigma,X)=(\sigma',X')$, where
$X'=\{A'_{i_1},A'_{i_2},\ldots,A'_{i_r}\}_<$.  By Theorem~\ref{thm:4}, we get
the following corollary.

\begin{cor}
The map $\ovxi:\ncnn(n)\to\ncna(n)$ is a bijection. Moreover, if $\ovxi(\sigma,X)=(\sigma',X')$, $X=\{A_1,\ldots,A_r\}_<$ and $X'=\{A'_1,\ldots,A'_s\}_<$, then $\type(\sigma)=\type(\sigma')$, $r=s$ and $|A_i|=|A_i'|$ for all $i\in[r]$. 
\end{cor}

\subsection{Rearranging nonnested blocks}

Let $(\sigma,X)\in\ncnn(n)$. Suppose $\{A_1,A_2,\ldots,A_\ell\}_<$ is the set of
all nonnested blocks of $\sigma$, $X=\{A_{i_1},A_{i_2},\ldots,A_{i_k}\}_<$, and
$\sigma_j = \sigma\cap [\min(A_j), \max(A_j)]$. Then we have $\sigma = \sigma_1
\uplus \sigma_2 \uplus \cdots \uplus \sigma_\ell$. For a permutation
$p=p_1p_2\cdots p_k$ of $[k]$, the \emph{rearrangement of $(\sigma,X)$ according
  to $p$} is defined to be the pair $(\sigma',X')$ of $\sigma' = \sigma_{a_1}
\uplus \sigma_{a_2} \uplus \cdots \uplus \sigma_{a_\ell}$ and
$X=\{A'_{i_1},A'_{i_2},\ldots,A'_{i_k}\}$, where $a_j = j$ if
$j\not\in\{i_1,i_2,\ldots,i_k\}$; and $a_j=i_{p_t}$ if $j = i_t$, and
$\{A'_1,A'_2,\ldots,A'_\ell\}_<$ is the set of all nonnested blocks of
$\sigma'$.

For $(\sigma,X)\in\ncnn(n)$ with $|X|=k$,
we define $\iota_B(\sigma,X)$ to be the rearrangement of $(\sigma,X)$ according to
$$p = \left\{
  \begin{array}{ll}
12\cdots k, 
& \mbox{if $k=2t$,}\\
(t+1) 1 2 \cdots t (t+2)(t+3) \cdots (2t+1),
& \mbox{if $k=2t+1$.}\\
 \end{array}\right.
$$

For $(\sigma,X,\epsilon)\in\ncnnd(n)$ with $|X|=k$,
we define $\iota_D(\sigma,X,\epsilon)$ to be $(\sigma',X',\epsilon)$, where $(\sigma',X')$ is the rearrangement of $(\sigma,X)$ according to
$$p = \left\{
  \begin{array}{ll}
12\cdots k, 
& \mbox{if $k=2t$ and $\epsilon=0$,}\\
t (t+1) 1 2 \cdots (t-1) (t+2)(t+3) \cdots (2t),
& \mbox{if $k=2t$ and $\epsilon\ne0$,}\\
 (t+1) 1 2 \cdots t (t+2)(t+3) \cdots (2t+1),
& \mbox{if $k=2t+1$.}\\
 \end{array}\right.
$$
Clearly, $\iota_B:\ncnn(n)\to\ncnn(n)$ and $\iota_D:\ncnnd(n)\to\ncnnd(n)$ are type-preserving bijections.

By the properties of the bijections we have defined so far, we get the following theorem.

\begin{thm}
The composed maps $(\nnmap_B)^{-1}\circ \ovrho \circ \ovxi \circ \iota_B\circ \ncmap_B$, $(\nnmap_C)^{-1}\circ \ovrho \circ \ovxi \circ \ncmap_B$ and $(\nnmap_D)^{-1}\circ \ovrho \circ \ovxi \circ \iota_D\circ \ncmap_D$ are type-preserving bijections between noncrossing partitions and nonnesting partitions of type $B_n$, $C_n$ and $D_n$ respectively; see Figures~\ref{fig:bijB} and \ref{fig:bijD}.
\end{thm}

\begin{remark}
  Our type-preserving bijections are different from those of Fink and Giraldo
  \cite{Fink} because our bijections do not preserve certain statistics
  preserved by their bijections. In fact, they showed that their bijections are
  the unique ones preserving those statistics. There are other bijections
  between noncrossing and nonnesting partitions of classical types due to Rubey
  and Stump \cite{Rubey2010} for type $B$ and Conflitti and Mamede
  \cite{Conflitti} for type $D$. However their bijections preserve not the types
  but `openers' and `closers'.
\end{remark}

\section{Another interpretation for noncrossing partitions of type $B$ and type $D$}
\label{sec:another_interpret}

We denote by $\ncbb(n)$ the set of pairs $(\sigma, x)$, where $\sigma\in\nca(n)$
and $x$ is either $\emptyset$, an edge or a block of $\sigma$. Note that if a
partition $\sigma$ of $[n]$ has $i$ edges, then there are $n-i$ blocks in
$\sigma$. For each $\sigma\in\nca(n)$, we have $n+1$ choices for $x$ with
$(\sigma,x)\in\ncbb(n)$. Hence, $\ncbb(n)$ is essentially the same as
$\nca(n)\times[n+1]$.

We define a map $\varphi_B:\ncnn(n)\rightarrow\ncbb(n)$ as follows.  For
$(\sigma,X)\in\ncnn(n)$ with $X=\{A_1,A_2,\ldots,A_k\}_<$, $\varphi_B(\sigma,X)$
is defined to be $(\sigma',x)$, where $\sigma'$ is the partition obtained from
$\sigma$ by unioning $A_i$ and $A_{k+1-i}$ for $i=1,2,\ldots,\floor{k/2}$, and
$$x = \left\{
  \begin{array}{ll}
    \emptyset, & \mbox{if $k=0$;} \\
    (\max(A_{t}), \min(A_{t+1})), & \mbox{if $k\ne0$ and $k=2t$;}\\
    A_{t+1}, & \mbox{if $k=2t+1$.}
 \end{array} \right. $$

\begin{example}
If $\sigma=\{\{1,2\}, \{3\}, \{4,7\}, \{5,6\}, \{8,9,10\}, \{11\}\}$ and
  $X=\{ \{1,2\}$, $\{3\}$, $\{4,7\}$, $\{8,9,10\}$, $\{11\}\}$, then
  $\varphi_B(\sigma,X)=(\sigma',x)$, where $\sigma'=\{\{1,2,11\}$,
  $\{3,8,9,10\}$, $\{4,7\}$, $\{5,6\}\}$ and $x$ is the block $\{4,7\}$.
\end{example}

\begin{thm}\label{thm:simbij}
  The map $\psi_B=\varphi_B \circ \ncmap_B$ is a bijection between $\ncb(n)$ and
  $\ncbb(n)$.  Moreover, if $\psi_B(\pi)=(\sigma,x)$, then
  $\type(\pi)=\type(\sigma)$ if $x$ is not a block; and
  $\type(\pi)=\type(\sigma\setminus\{x\})$ if $x$ is a block.
\end{thm}
\begin{proof}
Since $\ncmap_B:\ncb(n)\to\ncnn(n)$ is a bijection, it is sufficient to show that $\varphi_B:\ncnn(n)\rightarrow\ncbb(n)$ is a bijection. Let us find the inverse map of $\varphi_B$. 

Let $(\sigma,x)\in\ncbb(n)$. Then we construct $\sigma'$ and $X$ as follows.

If $x=\emptyset$, then $\sigma'=\sigma$ and $X=\emptyset$. 

If $x$ is an edge $(a,b)$, then let $E$ be the set of edges $(i,j)$ of $\sigma$
with $i\leq a<b\leq j$. Then $\sigma'$ is the partition obtained from $\sigma$
by removing the edges in $E$, and $X$ is the set of blocks of $\sigma'$ which
contain an endpoint of an edge in $E$. Here the endpoints of an edge $(i,j)$ are
the integers $i$ and $j$.

If $x$ is a block $B$, then let $E$ be the set of edges $(i,j)$ of $\sigma$ with $i< \min(B)\leq\max(B)< j$. Then $\sigma'$ is the partition obtained from $\sigma$ by removing the edges in $E$, and $X$ is the set of blocks of $\sigma'$ which are equal to $B$ or contain an endpoint of an edge in $E$.

It is easy to see that the map $(\sigma,x)\mapsto (\sigma',X)$ is the inverse of $\varphi_B$. The `moreover' statement is clear from the construction of $\ncmap_B$ and $\varphi_B$.
\end{proof}

Since $\ncbb(n)$ is the same as $\nca(n)\times[n+1]$, Theorem~\ref{thm:simbij} gives a bijective proof of $\#\ncb(n)=\binom{2n}{n}$.

\begin{remark}
  For $\pi\in\ncb(n)$, let $Abs(\pi)$ be the partition in $\nca(n)$ such that
  $B$ is a block of $Abs(\pi)$ if and only if $B=\{|i|:i\in B'\}$ for some
  $B'\in\pi$.  Biane et al.~\cite[Theorem in Subsection 14]{Biane2003} proved
  that the map $\pi\mapsto Abs(\pi)$ is an $(n+1)$-to-$1$ map from $\ncb(n)$ to
  $\nca(n)$, thus proved $\#\ncb(n)=\binom{2n}{n}$ bijectively. In fact, they
  proved that $\ncb(n)$ is in bijection with the set of pairs $(\sigma,x)$ where
  $\sigma\in\nca(n)$ and $x$ is a block of either $\sigma$ or the Kreweras
  complement $\mathrm{Kr}(\sigma)$. The Kreweras complement has the property
  that the sum of the number of blocks of $\sigma$ and the number of blocks of
  $\mathrm{Kr}(\sigma)$ is equal to $n+1$. It is easy to check that if
  $\varphi_B\circ\ncmap_B(\pi)=(\sigma,x)$, then $\sigma=Abs(\pi)$.
\end{remark}

We denote by $\ncdd(n)$ the set of pairs $(\sigma,x)$ such that
$\sigma\in\nca(n-1)$ and $x$ is either $\emptyset$, an edge of $\sigma$, a block
of $\sigma$ or an integer in $[\pm(n-1)]$. We can also easily see that
$\ncdd(n)$ is essentially the same as $\nca(n-1)\times[3n-2]$.

We define a map $\varphi_D:\ncnnd(n-1)\rightarrow\ncdd(n)$ as follows.  Let
$(\sigma,X,\epsilon)\in\ncnnd(n-1)$ and $X=\{A_1,A_2,\ldots,A_k\}_<$.  Then
$\varphi_D(\sigma,X,\epsilon)$ is defined to be $(\sigma',x)$, where $\sigma'$
is the partition obtained from $\sigma$ by unioning $A_i$ and $A_{k+1-i}$ for
$i=1,2,\ldots,\floor{k/2}$, and
$$x = \left\{
  \begin{array}{ll}
    \emptyset, & \mbox{if $\epsilon=0$ and $k=0$;} \\
    (\max(A_{t}), \min(A_{t+1})), & \mbox{if $\epsilon=0$, $k=2t\ne0$;}\\
    A_{t+1}, & \mbox{if $\epsilon=0$ and $k=2t+1$,}\\
    \epsilon \cdot \max(A_{\floor{(k+1)/2}}) & \mbox{if $\epsilon\ne0$.}
\end{array} \right. $$

\begin{thm}\label{thm:10}
  The map $\psi_D=\varphi_D \circ \ncmap_D$ is a bijection between $\ncd(n)$ and $\ncdd(n)$. 
Moreover, if $\psi_D(\pi)=(\sigma,x)$, then 
$$\type(\pi)=\left\{
\begin{array}{ll}
  \type(\sigma)\doublecup\{1\}, & \mbox{if $x=\emptyset$ or $x$ is an edge,}\\
  \type(\sigma\setminus\{x\}), & \mbox{if $x$ is a block,}\\
  \type(\sigma\setminus\{B\}) \doublecup\{|B|+1\}, & \mbox{if
    $x\in[\pm(n-1)]$ and $B$ is the block of $\sigma$ containing $|x|$.}
\end{array}\right.$$
\end{thm}
\begin{proof}
The proof is similar to that of Theorem~\ref{thm:simbij}, hence we omit it.
\end{proof}

Since $\ncdd(n)$ is the same as $\nca(n-1)\times[3n-2]$, Theorem~\ref{thm:10} gives a bijective proof of $\# \ncd(n) = \frac{3n-2}{n}\binom{2(n-1)}{n-1}$.

For an integer partition $\lambda=\{1^{m_1},2^{m_2},\ldots\}$, let
$m_{\lambda}=m_1!m_2!\cdots$.

Kreweras proved the following formula for the number of $\pi\in\nca(n)$ with given block sizes.

\begin{thm}[\cite{Kreweras1972}] \label{thm:kre}
Let $\lambda$ be an integer partition with $|\lambda|=n$ and length $\ell$. Then the number of $\pi\in\nca(n)$ with $\type(\pi)=\lambda$ is equal to 
$$ \frac{n!}{m_{\lambda} (n-\ell+1)!}.$$
\end{thm}

As an application of Theorems~\ref{thm:simbij} and \ref{thm:10}, we can give
another proof of the following type $B$ and type $D$ analogs of
Theorem~\ref{thm:kre}.

\begin{thm}[\cite{Athanasiadis1998}]\label{thm:type-enumeration-B}
Let $\lambda$ be an integer partition with $|\lambda|\leq n$ and length $\ell$. Then the number of $\pi\in\ncb(n)$ with $\type(\pi)=\lambda$ is equal to 
$$  \frac{n!}{m_{\lambda}(n-\ell)!}.$$
\end{thm}
\begin{proof}
Let $|\lambda|= n-k$ and $\psi_B(\pi)=(\sigma,x)\in\ncbb(n)$.

If $k=0$, then $\pi$ does not have a zero block and $x$ is not a block. Since
$\sigma$ has $\ell$ blocks and $n-\ell$ edges, there are $(n-\ell+1)\cdot
\frac{n!}{m_{\lambda} (n-\ell+1)!}=\frac{n!}{m_{\lambda}(n-\ell)!}$ choices of
$(\sigma,x)\in\ncbb(n)$.

If $k\ne0$, then $\pi$ has a zero block of size $2k$. Thus $x$ is a block of
size $k$ in $\sigma$. Let $\lambda=\{1^{m_1},2^{m_2},\ldots\}$ and
$\lambda'=\type(\sigma)$. Note that $\lambda'=\lambda\doublecup\{k\}$ and
$m_{\lambda'}=m_{\lambda} \cdot\frac{(m_k+1)!}{m_k!} = m_{\lambda} (m_k+1)$.
Thus, there are $\frac{n!}{m_{\lambda'} (n-\ell)!}$ choices for
$\sigma\in\nca(n)$ and for each $\sigma$ there are $(m_k+1)$ choices for
$x$. Thus we get the desired formula.
\end{proof}

\begin{thm}[\cite {Athanasiadis2005}]
  Let $\lambda=\{1^{m_1},2^{m_2},\ldots\}$ be an integer partition with
  $|\lambda|\leq n$ and length $\ell$. Then the number of $\pi\in\ncd(n)$ with
  $\type(\pi)=\lambda$ is equal to
$$ \left\{
    \begin{array}[display]{ll}
\displaystyle  
\frac{(n-1)!}{m_{\lambda}(n-\ell-1)!}, & \mbox{if $|\lambda|\leq n-2$,}\\
\displaystyle
 (m_1+2(n-\ell)) \frac{(n-1)!}{m_{\lambda}(n-\ell)!},
 & \mbox{if $|\lambda|=n$.}\\
    \end{array}
  \right.$$
\end{thm}
Note that if $\type(\pi)=\lambda$ for $\pi\in\ncd(n)$, then $|\lambda|$ can not be $n-1$.

\begin{proof}
Let $|\lambda|= n-k$ and $\psi_D(\pi)=(\sigma,x)$. 

If $k\geq2$, then $x$ is a block of size $k$ and we can use the same argument in the proof of Theorem~\ref{thm:type-enumeration-B}.

Assume $k=0$. Then $x$ is either $\emptyset$, an edge of $\sigma$ or an integer in $[\pm(n-1)]$.

If $x=\emptyset$, then
$\type(\sigma)=\lambda\setminus\{1\}=\{1^{m_1-1},2^{m_2},\ldots\}$.

If $x$ is an edge, then the type of $\sigma$ is $\lambda\setminus\{1\}$. Since
$\sigma$ has $\ell-1$ blocks, there are $n-\ell$ choices of $x$.

Let $\lambda'=\lambda\setminus\{1\}$. Then there are
$\frac{(n-1)!}{m_{\lambda'} ((n-1)-(\ell-1)+1)!}$ choices of $\sigma$ and $n-\ell+1$ choices of $x$.  Thus there are 
\begin{equation}
  \label{eq:4}
\frac{(n-1)!}{m_{\lambda'} (n-\ell)!}=m_1\cdot \frac{(n-1)!}{m_{\lambda} (n-\ell)!}
\end{equation}
possibilities when $x$ is either $\emptyset$ or an edge. 

Now assume that $x$ is an integer in $[\pm(n-1)]$. If $|x|$ is contained in a
block of size $i$, then the corresponding block in $\sigma$ is of size
$i+1$. Thus 
$$\type(\sigma)=\lambda^{(i)} = 
\{1^{m_1},\ldots,(i-1)^{m_{i-1}},i^{m_i +1}, (i+1)^{m_{i+1}-1}, (i+2)^{m_{i+2}},\ldots\}.$$
Note that $m_{\lambda^{(i)}}=m_{\lambda}\cdot\frac{1+m_i}{m_{i+1}}$.  Thus there
are $\frac{(n-1)!}{m_{\lambda^{(i)}} (n-1-\ell+1)!}$ choices of $\sigma$. For
each $\sigma$, there are $1+m_i$ choices for the block containing $x$, and $2i$
choices for $x$.
Thus in this case the number of possible $(\sigma,x)$'s is equal to
\begin{equation}
  \label{eq:1}
\sum_{i\geq1} 2i (1+m_1) \frac{(n-1)!}{m_{\lambda^{(i)}} (n-\ell)!}  =
\frac{2(n-1)!}{m_\lambda (n-\ell)!} \sum_{i\geq1} (1+m_i)\cdot \frac{i\cdot m_{i+1}} {1+m_i}.
\end{equation}

Since
\begin{align*}
\sum_{i\geq1} i\cdot m_{i+1} &=\sum_{i\geq0} i\cdot m_{i+1} 
= \sum_{i\geq0} (i+1) m_{i+1} -\sum_{i\geq0} m_{i+1}\\
  &= \sum_{i\geq1} i\cdot m_i - \sum_{i\geq1} m_i = n-\ell,
\end{align*}
\eqref{eq:1} is equal to $(n-\ell)\cdot \frac{2(n-1)!}{m_\lambda (n-\ell)!}$.
The sum of \eqref{eq:4} and \eqref{eq:1} gives the desired formula.
\end{proof}

\section{Lattice paths}
\label{sec:latticepaths}

Let $\LP(n)$ denote the set of lattice paths from $(0,0)$ to $(n,n)$ consisting
of up step $(0,1)$ and east step $(1,0)$. A \emph{Dyck path} of length $2n$ is a
lattice path in $\LP(n)$ which never goes below the line $y=x$.

It is well known that $\nca(n)$ is in bijection with the set of \emph{Dyck path}
of length $2n$: the Dyck path corresponding to $\sigma\in\nca(n)$ is determined
as follows.  The $(2i-1)$th step and the $(2i)$th step are, respectively,
$(0,1)$ and $(0,1)$ if $i$ is the minimum of a non-singleton block of $\sigma$;
$(1,0)$ and $(1,0)$ if $i$ is the maximum of a non-singleton block of $\sigma$;
$(0,1)$ and $(1,0)$ if $\{i\}$ is a block of $\sigma$; $(1,0)$ and $(0,1)$
otherwise.

Now let us find a bijection between $\ncb(n)$ and $\LP(n)$. Since $\ncb(n)$ is
in bijection with $\ncnn(n)$, we will use $\ncnn(n)$ instead of $\ncb(n)$.

Let $(\sigma,X)\in\ncnn(n)$. Suppose $P$ is the Dyck path corresponding to
$\sigma$. Consider a block $B\in X$ with $\min(B)=i$ and $\max(B)=j$. Since $B$
is nonnested, the $(2i-1)$th step starts at $(i-1,i-1)$ and the $(2j)$th step
ends at $(j,j)$.  Then we reflect the subpath of $P$ consisting of the $r$th
steps for all $r\in[2i-1,2j]$ across the line $y=x$. Let $g(\sigma,X)$ be the
lattice path obtained by this reflection for each $B\in X$.

\begin{example}
  Let $\sigma=\{\{1,4,5\},\{2,3\},\{6\},\{7,9\},\{8\},\{10\}\}$ and
  $X=\{\{1,4,5\}$, $\{6\}$, $\{10\}\}$. Then $(\sigma,X)\in\ncnn(10)$. The
  lattice path $g(\sigma,X)$ is obtained from the Dyck path corresponding to
  $\sigma$ by reflecting the subpaths corresponding to the nonnested blocks in
  $X$. See Figure~\ref{fig:path}.
\end{example}

It is easy to see that the map $g$ is a bijection.
\begin{prop}
  The map $g:\ncnn(n)\to\LP(n)$ is a bijection.
\end{prop}

Thus we get $\# \ncb(n) = \#\ncnn(n) = \binom{2n}{n}$. Note that we did not use
the number of Dyck paths. Since $\#\ncb(n) = \#\ncbb(n) = (n+1)\cdot \#\nca(n)$,
we get another combinatorial proof of the fact that the number of Dyck paths of
length $2n$ is equal to the Catalan number $\frac{1}{n+1}\binom{2n}{n}$.

\begin{remark}
Reiner \cite[Proposition~17]{Reiner1997} also found a bijection between $\ncb(n)$ and $\LP(n)$ which is different from ours. Ferrari \cite[Proposition~2.5]{Ferrari} considered the set $\widetilde{\operatorname{NC}}(n)$ of `component-bicoloured' noncrossing partitions of $[n]$ and found a bijection between this set and $\LP(n)$. In fact, $\widetilde{\operatorname{NC}}(n)$ is essentially the same as $\ncnn(n)$ and our bijection $g$ is identical with Ferrari's bijection. 
\end{remark}

\begin{figure}
  \centering
 \begin{pspicture}(0,0)(21,10)
    \psgrid[gridcolor=gray,gridlabels=0pt,subgriddiv=1](0,0)(10,10)
    \psline(0,0)(10,10)
   \psline[linewidth=2pt](0,0)(0,4)(3,4)(3,5)(5,5)(5,6)(6,6)(6,9)(9,9)(9,10)(10,10)
    \rput(11,5){$\Rightarrow$}
    \rput(12,0){    \psgrid[gridcolor=gray,gridlabels=0pt,subgriddiv=1](0,0)(10,10)
    \psline(0,0)(10,10)
   \psline[linewidth=2pt](0,0)(4,0)(4,3)(5,3)(5,5)(6,5)(6,6)(6,9)(9,9)(10,9)(10,10)}
 \end{pspicture}
\caption{A lattice path is obtained from a Dyck path by reflecting several subpaths.}
  \label{fig:path}
\end{figure}
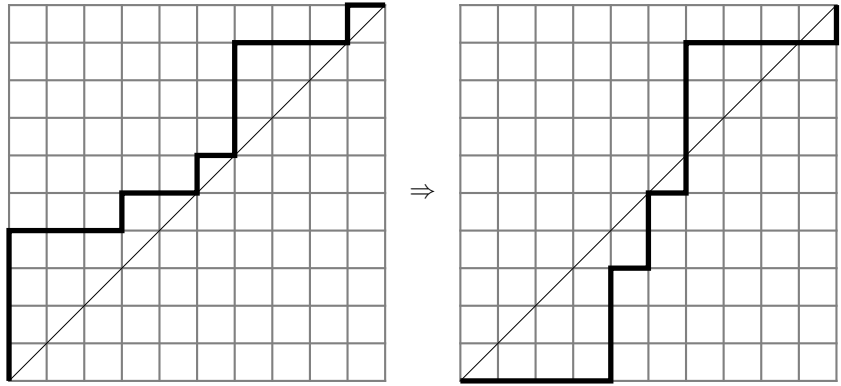

We can also find a bijection between $\ncd(n)$ and a subset of $\LP(n)$. To do
this, we need another interpretation for $\ncd(n)$.

We denote by $\ovncnn(n)$ the set of elements $(\sigma,X)\in\ncnn(n)$ such that
if $X$ has a block $A$ containing $n$, then $|A|\geq2$.

For $(\sigma,X,\epsilon)\in\ncnnd(n-1)$ with $X=\{A_1,A_2,\ldots,A_k\}_<$, we
define $\kappa(\sigma,X,\epsilon)$ to be the pair $(\sigma',X')$, where
$\sigma'$ and $X'$ are defined as follows:
\begin{itemize}
\item If $\epsilon =0$, then let $\sigma'$ be the partition obtained from $\sigma$ by adding the singleton $\{n\}$ and let $X'=X$. 
\item If $\epsilon =1$, then let $\sigma'$ be the partition obtained from $\sigma$ by adding $n$ to the block $A_k$ and let $X'=X$.
\item If $\epsilon =-1$, then let $\sigma'$ be the partition obtained from $\sigma$ by adding $n$ to the block $A_k$ and let $X'=X\setminus\{A_k\}$.
\end{itemize}

One can easily check that this is a bijection.

\begin{prop}
The map $\kappa:\ncnnd(n-1)\to\ovncnn(n)$ is a bijection.
\end{prop}

Let $\ovLP(n)$ denote the set of lattice paths in $\LP(n)$ which do
not touch $(n-1,n-1)$ and $(n,n-1)$ simultaneously. Note that the cardinality of
$\ovLP(n)$ is equal to $\binom{2n}{n} - \binom{2n-2}{n-1}$.
It is easy to see that $g(\sigma,X)\in\ovLP(n)$ for each $(\sigma,X)\in\ovncnn(n)$, and the map $g:\ovncnn(n)\to\ovLP(n)$ is a bijection.

\begin{prop}
 The map $g:\ovncnn(n)\to\ovLP(n)$ is a bijection.
\end{prop}

Thus we get a combinatorial proof of $\#\ncd(n)  = \#\ovncnn(n) = \binom{2n}{n} - \binom{2n-2}{n-1}$.

\section{Catalan tableaux of classical types} 
\label{sec:catal-tabl-class}

A \emph{Ferrers diagram} is a left-justified arrangement of square cells with
possibly empty rows and columns. The \emph{length} of a Ferrers diagram is the
sum of the number of rows and the number of columns. If a Ferrers diagram is of
length $n$, then we label the steps in the border of the Ferrers diagram with
$1,2,\ldots,n$ from north-west to south-east. We label a row (resp.~column) with
$i$ if the row (resp.~column) contains the south (resp.~east) step labeled with
$i$. The \emph{$(i,j)$-entry} is the cell in the row labeled with $i$ and in the
column labeled with $j$. See Figure~\ref{fig:ferres}.

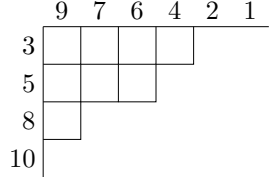
\begin{figure}
  \centering
  \begin{pspicture}(0,1)(6,-4) 
\psline(0,0)(6,0) \psline(0,0)(0,-4)
\rownum[0,3] \rownum[1,5] \rownum[2,8] \rownum[3,10]
\colnum[0,9] \colnum[1,7] \colnum[2,6] \colnum[3,4] \colnum[4,2] \colnum[5,1]
\cell(1,1)[] \cell(1,2)[] \cell(1,3)[] \cell(1,4)[] \cell(2,1)[] \cell(2,2)[] \cell(2,3)[] \cell(3,1)[]     
  \end{pspicture}
\caption{A Ferrers diagram with labeled rows and columns.}
  \label{fig:ferres}
\end{figure}

For a Ferrers diagram $F$, a \emph{permutation tableau} of shape $F$ is a
$0,1$-filling of the cells in $F$ satisfying the following conditions:
\begin{enumerate}
\item each column has at least one $1$,
\item there is no $0$ which has a $1$ above it in the same column and a $1$ to the left of it in the same row. 
\end{enumerate}

The \emph{length} of a permutation tableau is defined to be the length of its
shape.  A \emph{Catalan tableau} is a permutation tableau which has exactly one
$1$ in each column. Let $\cta(n)$ denote the set of Catalan tableaux of length
$n$. There is a simple bijection between $\cta(n)$ and $\nca(n)$ due to Burstein
\cite[Theorem~3.1]{Burstein2007}. His bijection can be described in the
following way which is similar to that in the proof of Proposition~6 in
\cite{Corteel2009}.

Let $\sigma\in\nca(n)$. We first make the Ferrers diagram $F$ as follows. The
$i$th step of the border of $F$ is south if $i$ is the smallest integer in the
block containing $i$; and west otherwise. We fill the $(i,j)$-entry with $1$ if
and only if $i$ and $j$ are in the same block whose smallest integer is $i$. One
can easily check that this is a bijection. For more information of Catalan
tableaux and permutation tableaux, see \cite{Steingrimsson2007,Viennot}.

Lam and Williams \cite{Lam2008a} defined permutation tableaux of type $B_n$. See
\cite{Nadeau} for the `alternative tableaux' version.  The definition of
permutation tableaux of type $B_n$ in \cite{Lam2008a} can be written as follows.

Let $F$ be a Ferrers diagram with $k$ columns including empty columns. The
\emph{shifted} Ferrers diagram $\overline F$ of $F$ is the diagram obtained from
$F$ by adding $k$ rows of size $1,2,\ldots,k$ above it in increasing order. The
rightmost cell of an added row is called \emph{diagonal}. We label the added
rows as follows. If the diagonal of an added row is in the column labeled with
$i$, then the row is labeled with $-i$. For example, see Figure~\ref{fig:catb};
at this moment, ignore the $0$'s and $1$'s.

A \emph{permutation tableau of type $B_n$} is a $0,1$-filling of the cells in
the shifted Ferrers diagram $\overline F$ for a Ferrers diagram $F$ of length
$n$ satisfying the following conditions:
\begin{enumerate}
\item each column has at least one $1$,
\item there is no $0$ which has a $1$ above it in the same column and a $1$ to the left of it in the same row,
\item if a $0$ is in a diagonal, then it does not have a $1$ to the left of it in the same row. 
\end{enumerate}

A \emph{Catalan tableau of type $B_n$} is a permutation tableau of type $B_n$
such that each column has exactly one $1$.  A \emph{Catalan tableau of type
  $D_n$} is a Catalan tableau of type $B_n$ with the following additional
condition: if the last row is not empty, then the left most column does not have
$1$ in the topmost cell.  Let $\ctb(n)$ and $\ctd(n)$ denote the set of Catalan
tableaux of type $B_n$ and type $D_n$ respectively.

Now we will find a bijection between $\ncnn(n)$ and $\ctb(n)$. 

Let $(\sigma,X)\in\ncnn(n)$. Suppose $F$ is the Ferrers diagram of length $n$
such that the $i$th step of the border of $F$ is south if $i$ is the smallest
integer in a block of $\sigma$ which is not in $X$; and west otherwise. Let $T$
be the 0,1-filling of the shifted Ferrers diagram $\overline F$ obtained as
follows.  For each $i$ which is the smallest integer in a block in $X$, fill the
$(-i,i)$-entry with $1$. For each pair $(i,j)$ of distinct integers such that
$i$ and $j$ are in the same block $B$ and $i=\min(B)$, fill the $(-i,j)$-entry
with $1$ if $B$ is in $X$ ; and fill the $(i,j)$-entry with $1$ otherwise. Fill
the remaining entries with $0$'s. We define $f(\sigma,X)$ to be $T$.  For
example, see Figure~\ref{fig:catb}.
 
\begin{figure}
  \centering
  \begin{pspicture}(0,1)(6,-10) 
    \psline(0,0)(0,-10)
\multido{\n=0+1}{6}{\put(\n,-\n){\psline(0,0)(1,0)}}
    \cell(1,1)[] \rput(.5,-.5){$0$}
    \cell(2,1)[0] \cell(2,2)[0] \cell(3,1)[0] \cell(3,2)[0] 
    \cell(3,3)[0] \cell(4,1)[1] \cell(4,2)[1] \cell(4,3)[0]
    \cell(4,4)[1] \cell(5,1)[0] \cell(5,2)[0] \cell(5,3)[0]
    \cell(5,4)[0] \cell(5,5)[0] \cell(6,1)[0] \cell(6,2)[0] \cell(6,3)[0]
    \cell(6,4)[0] \cell(6,5)[1] \cell(6,6)[1] \cell(7,1)[0]
    \cell(7,2)[0] \cell(7,3)[0] \cell(7,4)[0] \cell(8,1)[0] \cell(8,2)[0]
    \cell(8,3)[1] \cell(9,1)[0] \rownum[0,-9] \rownum[1,-7]
    \rownum[2,-6] \rownum[3,-4] \rownum[4,-2] \rownum[5,-1] \rownum[6,3]
    \rownum[7,5] \rownum[8,8] \rownum[9,10]
    \colnum[0,9] \colnum[1,7] \colnum[2,6] \colnum[3,4] \colnum[4,2] \colnum[5,1] 
  \end{pspicture}
  \caption{The Catalan tableau $f(\pi,X)$ of type $B_{10}$ for $\pi=\{\{1,2\},
    \{3\}, \{4,7,9\}, \{5,6\}, \{8\}, \{10\}\}$ and $X=\{\{1,2\},
    \{4,7,9\}\}$.}
  \label{fig:catb}
\end{figure}

\begin{thm}
The map $f$ is a bijection between $\ncnn(n)$ and $\ctb(n)$.
\end{thm}
\begin{proof}
  First, we will show that $T=f(\sigma,X)\in \ctb(n)$.  By the construction,
  each column of $T$ contains exactly one $1$, and the row of $T$ labeled with
  $-i$ has a $1$ if and only if the diagonal entry in the row is filled with
  $1$. To prove $T\in \ctb(n)$, it only remains to show that there is no $0$
  which has a $1$ above it in the same column and a $1$ to the left of it in the
  same row. Since each column has only one $1$, this condition is equivalent to
  the following: there is no quadruple $(i,j,i',j')$ with $i<i'$, $j<j'$ and
  $|i'|<j$ such that both the $(i,j)$-entry and the $(i',j')$-entry are filled
  with $1$, where $i$ and $i'$ can be negative. Note that we also have $|i|\leq
  j$ and $|i'|\leq j'$ because there are the $(i,j)$-entry and the
  $(i',j')$-entry.

  Suppose that we have such a quadruple $(i,j,i',j')$. Then we have either
  $|i|<|i'|<j<j'$ or $|i'|<|i|\leq j<j'$. Let $B$ and $B'$ be the blocks of
  $\sigma$ with $|i|,j\in B$ and $|i'|,j'\in B'$. If $|i|<|i'|<j<j'$, then we
  must have $B=B'$ since $\sigma\in\nca(n)$. Then $|i|=\min(B)=\min(B')=|i'|$,
  which is a contradiction. If $|i'|<|i|\leq j<j'$, then $i<0$. Thus $B$ is in
  $X$, which implies that $B$ is nonnested. However this is a contradiction
  because $|i'|<|i|\leq j<j'$ and $\sigma\in\nca(n)$, $B$ cannot be
  nonnested.

Now we define the inverse map of $f$. Let $T\in\ctb(n)$. Define $\sigma$ to be the partition of $[n]$ such that $i$ and $j$ are in the same block $B$ with $\min(B)=i$ if and only if $i<j$ and
either the $(i,j)$-entry or the $(-i,j)$-entry of $T$ is filled with $1$. Define $X$ to be the set of blocks $B$ of $\sigma$ such that the row of $T$ labeled with $-\min(B)$ contains a $1$. It is easy to see that the map $T\mapsto(\sigma,X)$ is the inverse of $f$.
\end{proof}

\begin{remark}
  Burstein's bijection between $\cta(n)$ and $\nca(n)$ in \cite{Burstein2007} is
  a restriction of the `zigzag' map for permutation tableaux in
  \cite{Steingrimsson2007}. We will not go into the details but our map $f$ can
  also be expressed as a restriction of a type $B$ analog of the `zigzag'
  map.
\end{remark}

If we restrict $f$ to $\ovncnn(n)$, we get the following theorem.

\begin{thm}
The map $f:\ovncnn(n)\to\ctd(n)$ is a bijection.
\end{thm}

\section{Concluding remarks}

\begin{figure}
$$
\psset{nodesep=3pt, shortput=nab}
\psmatrix
& \ctb(n) & \ncna(n) & \nnna(n) & \nnc(n) \\
\ncb(n) & \ncnn(n) & \ncnn(n) & \ncna(n) & \nnna(n) \\
& \LP(n) & \ncbb(n) & \nca(n)\times[n+1] & \nnb(n)
\ncline{1,2}{2,2}<{f}
\ncline{2,1}{2,2}^{\ncmap_B}
\ncline{2,2}{3,2}<{g}
\ncline{2,2}{1,3}^{\ovxi}
\ncline{2,2}{2,3}^{\iota_B}
\ncline{2,2}{3,3}^{\varphi_B}
\ncline{3,3}{3,4}^{=}
\ncline{2,3}{2,4}^{\ovxi}
\ncline{2,4}{2,5}^{\ovrho}
\ncline{2,5}{3,5}>{\nnmap_B}
\ncline{1,3}{1,4}^{\ovrho}
\ncline{1,4}{1,5}^{\nnmap_C}
\endpsmatrix
$$
\caption{Bijections from $\ncb(n)$.}
  \label{fig:bijB}
\end{figure}

\begin{figure}
$$
\psset{nodesep=3pt, shortput=nab}
\psmatrix
& \ncnnd(n-1) & \ncnad(n-1) & \nnnad(n-1) \\
\ncd(n) & \ncnnd(n-1) & \ncdd(n) & \nnd(n) \\
\ctd(n) & \ovncnn(n) & \ovLP(n) & \nca(n-1)\times[3n-2]
\ncline{1,2}{2,2}<{\iota_D}
\ncline{2,1}{2,2}^{\ncmap_D}
\ncline{3,1}{3,2}^{f}
\ncline{3,2}{3,3}^{g}
\ncline{2,2}{3,2}<{\kappa}
\ncline{2,2}{2,3}^{\varphi_D}
\ncline{2,3}{3,4}^{=}
\ncline{1,2}{1,3}^{\ovxi}
\ncline{1,3}{1,4}^{\ovrho}
\ncline{1,4}{2,4}>{\nnmap_D}
\endpsmatrix
$$
\caption{Bijections from $\ncd(n)$.}
  \label{fig:bijD}
\end{figure}

Figures~\ref{fig:bijB} and \ref{fig:bijD} illustrate the objects and the
bijections between them in this paper. We have two interpretations $\ncnn(n)$
and $\ncbb(n)$ for $\ncb(n)$. Since both of them are closely related to
$\nca(n)$, they may be useful to prove type $B$ analogs of interesting
properties of $\nca(n)$. In the author's sequel paper \cite{Kim}, the
interpretation $\ncbb(n)$ is used to study the poset structure of $\ncb(n)$ and
$\ncd(n)$. 

Since we have a bijection between $\ncb(n)$ and $\ncbb(n)
=\nca(n)\times[n+1]$, one can ask the following question.

\begin{question}
  Is there a natural bijection between $\nnb(n)$ and $\nna(n)\times[n+1]$?
\end{question}

\section*{Acknowledgement}
The author would like to thank Philippe Nadeau and Lauren Williams for their helpful comments and discussions.

% \bibliographystyle{abbrv}
% \bibliography{/Users/jskim/bongs/math/research/bib/ref}

\end{document}